\documentclass[11pt]{article}
\usepackage{amssymb}

\usepackage{latexsym}
\usepackage{amsmath}
\usepackage{amsthm}
\usepackage{mathtools}
\usepackage{epsfig}
\usepackage{amscd}
\usepackage{graphicx}
\usepackage{tikz-cd}
\usepackage{appendix}
\usepackage{enumerate}
\usepackage{color}
\usepackage{todonotes}

\usepackage{tikz}
\usetikzlibrary{decorations.markings}
\usetikzlibrary{arrows.meta}

\usepackage{extarrows}

\usepackage[colorlinks=true,linkcolor=blue,citecolor=blue]{hyperref}

\usepackage[all,cmtip]{xy}

\theoremstyle{plain}
\newtheorem{theorem}{Theorem}[section]
\newtheorem{proposition}[theorem]{Proposition}
\newtheorem{lemma}[theorem]{Lemma}

\newtheorem{conjecture}[theorem]{Conjecture}

\theoremstyle{definition} 
\newtheorem{definition}[theorem]{Definition}
\newtheorem{example}[theorem]{Example}

\theoremstyle{remark} 
\newtheorem{remark}[theorem]{Remark}

\numberwithin{equation}{section}

\newcommand{\cB}{\mathcal{B}}

\newcommand{\sgn}{\mathrm{sgn}}
\newcommand{\R}{\mathbb{R}}

\newcommand{\Sc}{\mathrm{Sc}}
\newcommand{\C}{\mathbb{C}}
\newcommand{\supp}{\mathrm{supp}}

\newcommand{\Z}{\mathbb{Z}}
\newcommand{\N}{\mathbb{N}}

\newcommand{\beq}[1]{\begin{equation} \label{#1}}
\newcommand{\eeq}{\end{equation}}

\DeclareMathOperator{\Spec}{spec}
\DeclareMathOperator{\Spin}{Spin}
\DeclareMathOperator{\ind}{Ind}
\DeclareMathOperator{\id}{id}

\DeclareMathOperator{\Real}{Re}
\DeclareMathOperator{\vol}{vol}

\begin{document}
	\title{Quantitative partitioned index theorem and noncompact band-width}

	\author{Peter Hochs\footnote{Institute for Mathematics, Astrophysics and Particle Physics, Radboud University, Nijmegen, the Netherlands, \texttt{{p.hochs@math.ru.nl}}} {} and Jinmin Wang\footnote{Institute of Mathematics, Chinese Academy of Sciences, Beijing, China,  \texttt{jinmin@amss.ac.cn}}}
%

		\maketitle
	
	\begin{abstract}
 Gromov's band-width conjecture gives a precise upper bound for the width of a compact Riemannian band with positive scalar curvature lower bound, assuming that 
  the cross-section of the band admits no positive scalar curvature metrics. Versions of this were proved by Cecchini and by Zeidler.  In this paper, we develop a  quantitative version of partitioned manifold index theory, which applies to noncompact hypersurfaces. Using this, we prove a version of Gromov's band-width estimate for possibly noncompact Riemannian bands.
	\end{abstract}

\tableofcontents

\section{Introduction}

In 2018, Gromov made the following conjecture, see Conjecture C on page 722 of \cite{Gromov18}.
\begin{conjecture}[Band-width conjecture; Gromov, 2018]\label{con band width cpt}
Let $N$ be a compact manifold without boundary, of dimension $n-1 \geq 5$. Suppose that $N$ does not admit Riemannian metrics of positive scalar curvature. Consider a Riemannian metric $g$ on $N \times [-1,1]$. Let $\ell$ be the distance between $N \times \{-1\}$ and $N \times \{1\}$ with respect to this metric. Then
\[
	\inf_{x\in N \times [-1,1]}\Sc_g(x)\le\frac{4\pi^2(n-1)}{n\ell^2}.
\]
\end{conjecture}
Here we write $\Sc_g$ for the scalar curvature associated to a Riemannian metric $g$. 
Gromov noted that the bound $\frac{4\pi^2(n-1)}{n\ell^2}$ is optimal \cite[pages 653--654]{Gromov18}. 

Versions of Gromov's conjecture were proved by Zeidler \cite[Theorem 1.4]{Zeidler22}  and Cecchini  \cite[Theorem D]{Cecchini20} (who achieved the optimal bound), under the condition that  Rosenberg index of $N$ is nonzero. The Rosenberg index is a refined obstruction to Riemannian metrics of positive scalar curvature on compact $\Spin$ manifolds \cite{Rosenberg83, Rosenberg86a, Rosenberg86b}. Zeidler extended these results to more general forms of Riemannian bands.  A consequence of \cite[Theorem 1.4]{Zeidler20} is the following.
\begin{theorem}[Zeidler, 2020]\label{thm band width cpt}
Let $X$ be a complete Riemannian $\Spin$-manifold of dimension $n$, and let $N \subset X$ be a compact hypersurface. Suppose that 
\begin{itemize}
\item $N$ has trivial normal bundle;
\item the map $\pi_1(N) \to \pi_1(X)$ induced by the inclusion is injective;
\item the Rosenberg index of $N$ in $KO_{n-1}(C^*(\pi_1(N)))$ is nonzero. 
\end{itemize}
Suppose  that $M$ is a neighbourhood of $N$ of width $\ell$. Then
\[
	\inf_{x\in M}\Sc_g(x)\le\frac{4\pi^2(n-1)}{n\ell^2}.
\]
\end{theorem}
In the special case where $X = N \times \R$, Theorem \ref{thm band width cpt} becomes \cite[Theorem D]{Cecchini20} which states that Conjecture  \ref{con band width cpt} is true if the Rosenberg index of $N$ is nonzero. (See \cite[Corollary 1.5]{Zeidler20}.)

The proofs of the results in \cite{Cecchini20, Zeidler20, Zeidler22}  involve the \emph{partitioned manifold index} of Dirac operators, and the closely related index of \emph{Callias-type Dirac operators} on noncompact manifolds. In this paper, we develop a  \emph{quantitative partitioned manifold index}, and use it to extend Theorem \ref{thm band width cpt} to noncompact $N$.

If $X$ is an odd-dimensional, complete Riemannian manifold, partitioned into two parts by a \emph{compact} hypersurface $N \subset X$, then Roe's partitioned manifold index is an integer $\ind(D, N)$ associated to a Dirac operator $D$ on $X$. Roe proved a \emph{partitioned manifold index theorem}, stating that
\beq{eq PMIT intro}
\ind(D, N) = \ind(D_N),
\eeq
for a Dirac operator $D_N$ on $N$ naturally associated to $D$. See \cite{Roe1988dualtoeplitz} for the original result, or  \cite{Higson91} and \cite[Thm. 4.4]{Roe1996indextheory} for later, more streamlined accounts. As immediate consequences, one obtains cobordism invariance of the index of Dirac operators on compact manifolds, obstructions to positive scalar curvature, and homotopy invariants of certain higher signatures. (See \cite[pp.\ 31--32]{Roe1996indextheory}.)

Several variations on and generalizations of Roe's partitioned manifold index theorem have been obtained over the years, see for example \cite{Karami2019relative-partitioned, Ludewig2025large, Schick2018largescale, Seto2018toeplitz, Zadeh2010indextheory}. For a version of Theorem \ref{thm band width cpt} where $N$ may be noncompact, we develop and apply a partitioned manifold index theorem for noncompact hypersurfaces. In that setting, several results were obtained recently: 
\cite[Theorem 4.44]{bunke2024coronas}, \cite[Corollary 7.6]{Engel2025relative}, \cite[Theorem 2.23]{HochsdeKok25}. 

For our purposes, the construction of the partitioned manifold index in \cite{HochsdeKok25} is most useful, as this construction involves a Callias-type operator, which is suitable for the quantitative estimates in terms of positive scalar curvature that we will use. A disadvantage of the result in \cite{HochsdeKok25} is that  the Riemannian distance on $N$ is assumed to be equivalent, in the sense of coarse geometry, to the restriction of the Riemannian distance on $X$. 

To overcome this, we refine the result in \cite{HochsdeKok25} to a \emph{quantitative partitioned manifold index theorem}. In the results in \cite{bunke2024coronas, Engel2025relative, HochsdeKok25}, the equality \eqref{eq PMIT intro} holds in the $K$-theory of a localized Roe algebra. Roughly speaking, this is the inductive limit of the Roe algebras of $r$-neighborhoods of $N$, as $r \to \infty$. (See Definition \ref{def:propagation}.) In our quantitative version, we obtain a version of \eqref{eq PMIT intro} in a specific neighborhood of $N$, which we can control. This is Theorem \ref{thm:quantitativeIndex}, which is an equality  of the form
\beq{eq quant ind thm intro}
\ind_q(D, N) = (i_N)_*(\ind(D_N) ) \in K_*(C^*(M)),
\eeq
for a specific type of closed set $M\subset X$ containing $N$ in its interior, where $i_N \colon N \to M$ is the inclusion map and $\ind_q$ is a quantitative version of the partitioned manifold index (see Definitions \ref{def quant index odd} and \ref{def quant index even}). We illustrate the difference between partitioned manifold index theory for noncompact hypersurfaces as in \cite{HochsdeKok25}  and our quantitative version
 in Example \ref{ex quant index thm}. 



In Theorem \ref{thm band width cpt}, the index-theoretic assumptions that the induced map on fundamental groups is injective and that the Rosenberg index is nonzero guarantee that the cross-section $N$ does not admit a positive scalar curvature metric. In the setting of a complete, possibly noncompact manifold, a natural analogue is the non-vanishing of the coarse index of the Dirac operator on $N$ in $K_*(C^*(N))$. However, as we show in Subsection \ref{sec cond 1 necessary},  this condition alone is not sufficient to deduce the band-width estimate, and an additional assumption describing how $N$ lies inside the band $M$ is required. More precisely, we prove the following noncompact version of Theorem \ref{thm band width cpt}, which is the main result in this paper. (This is Theorem \ref{thm:bandwidthForBand} in the main text.)
\begin{theorem}[Band-width inequality for possibly noncompact bands]\label{thm band width intro}
	Let $(M^n,g)$ be a regular $\Spin$ Riemannian band with boundary $\partial_- M \sqcup \partial_+ M$. 
	Let $\ell$ denote the distance between $\partial_-M$ and $\partial_+M$ in $M$.
	Suppose that
	\begin{enumerate}
		\item the identity map
		\[
		\id\colon (\partial_-M,d_{g_{\partial_-M}})
		\to (\partial_-M,d_g|_{\partial_- M})
		\]
		is a coarse equivalence, and
		\item the Dirac operator $D_{\partial_-M}$ has non-zero index in $K_{n-1}(C^*(\partial_-M))$.
	\end{enumerate}
	Then
\[
	\inf_{x\in M}\Sc_g(x)\le\frac{4\pi^2(n-1)}{n\ell^2}.
\]
\end{theorem}
The definition of regular $\Spin$ Riemannian bands is given in Definition \ref{def:band}. In particular, there always exists a complete Riemannian manifold $X$ containing $M$. In order to avoid involving the distance function on $X$, we employ the quantitative partitioned manifold theory and localize the partitioned manifold index to $M\subset X$. Then we use the equality \eqref{eq quant ind thm intro} to prove Theorem \ref{thm band width intro}.

\subsection*{Acknowledgements}

PH was supported by the Netherlands Organization for Scientific Research NWO, through grants OCENW.M.21.176 and OCENW.M.23.063. JW was partially supported by NSFC 12501169.

\section{A band-width inequality for noncompact bands}

\subsection{Uniformly embedded submanifolds}\label{sec:submanifolds}

We recall several geometric regularity conditions for noncompact manifolds.

Let $(X^n,g)$ be a complete Riemannian manifold. Assume that $(X,g)$ has \emph{bounded geometry}, i.e.\ 
\begin{enumerate}
	\item the global injectivity radius $\textup{inj}_g=\inf_{x\in X} \textup{inj}_g(x)$ is positive, and
	\item for every $k\geq 0$, the $k$-th covariant derivatives of the Riemannian curvature tensor are uniformly bounded.
\end{enumerate}

In any metric space $(Y, d_Y)$, for $y \in Y$ and $r>0$, we denote the open ball around $y$ of radius $r$ by $B_r(y)$. For $A\subset Y$, we write
\[
\begin{split}
{B}_r(A) &:= {\bigcup_{y \in A} B_r(y)};\\
\bar{B}_r(A) &:= \overline{B_r(A)}
\end{split}
\]
where the line on the right hand side denotes the closure.
We often use this closure, so that the sets $\bar B_r(\mathcal{N}_{\psi})$ used in Definition \ref{def quant index odd} are closed, so they are proper metric spaces and hence have well-behaved Roe algebras. 

The following notion is taken from \cite[Section~2.3]{Normallyhyperbolicinvariantmanifolds}.

\begin{definition}\label{def:uniformlyembedded}
	Let $N^{\,n-1}$ be a complete hypersurface in a complete Riemannian manifold $(X^n,g)$. We say that $N$ is \emph{uniformly embedded} in $X$ if there exists $\delta>0$ such that, for every $x\in N$, the intersection $B_\delta(x)\cap N$ can be represented, in normal coordinates at $x$, as the graph of a function
	\[
	h\colon T_xN \longrightarrow (T_xN)^\perp,
	\]
	whose derivatives of all orders are uniformly bounded.
\end{definition}

It is easy to see that a uniformly embedded hypersurface $(N^{n-1},g_N=g|_N)$ is itself a complete Riemannian manifold of bounded geometry. Moreover, any uniformly embedded hypersurface admits a uniform tubular neighborhood. As a standing assumption, we suppose that all manifolds we consider are oriented. In particular, this means that all hypersurfaces have trivial normal bundles.
\begin{theorem}[{\cite[Theorem~2.31]{Normallyhyperbolicinvariantmanifolds}}]\label{thm:uniformTubular}
	If $N^{n-1}\subset X^n$ is a uniformly embedded hypersurface, then there exist constants $\varepsilon>0$ and $C>0$ such that the following hold.
	\begin{enumerate}
		\item The map 
		\[
		F\colon N \times (-\varepsilon,\varepsilon) \longrightarrow  X,\qquad (x,t)\mapsto \exp_x(t n_x),
		\]
		is a diffeomorphism onto its image $U$, where $n_x$ denotes the unit normal vector to $N$ in $X$ at $x$.
		\item The metric $g$ on $U$ satisfies
		\[
		C^{-1}\, F_*(dt^2+g_N)\ \leq\ g|_U\ \leq\ C\, F_*(dt^2+g_N).
		\]
	\end{enumerate}
\end{theorem}
	We call a set $U$ as in Theorem \ref{thm:uniformTubular} a \emph{uniform geodesic normal neighborhood} of $N\subset X$.

We remark that, although the metric tensors $g|_U$ and $dt^2 + g_N$ appear similar on $U$, the Riemannian distance function on $X$ restricted on $U$ is not necessarily bi-Lipschitz equivalent to the product distance function. This subtle distinction plays an essential role in our main theorems, and we will provide explicit examples to illustrate it (see Example \ref{ex quant index thm} and Subsection \ref{sec cond 1 necessary}).

\subsection{The main result}

Our goal in this paper is to prove a noncompact version of the band-width inequality for Riemannian bands.

\begin{definition}\label{def:band}
	Let $(M^n,g)$ be a connected Riemannian manifold with boundary 
	\(\partial M=\partial_-M\sqcup\partial_+M\), where $\partial_{\pm}M$ are unions of connected components of $\partial M$.
	We say that $(M^n,g)$ is a \emph{regular Riemannian band} if it admits an isometric embedding into a complete Riemannian manifold $(X,g)$ with bounded geometry, such that $\partial_\pm M$ are uniformly embedded hypersurfaces.
\end{definition}

\begin{remark}
	If $(M,g)$ has bounded geometry in the sense of manifolds with boundary, and $\partial_\pm M$ are uniformly embedded hypersurfaces (analogous to Definition~\ref{def:uniformlyembedded}), then $(M,g)$ is automatically regular: one may take
	\[
	X=(\partial_-M\times(-\infty,0])\cup_{\partial_-M} 
	M
	\cup_{\partial_+M} (\partial_+M\times[0,+\infty)),
	\]
	with a metric smoothly interpolating between $g$ and the product metrics near $\partial_\pm M$ as in \cite[Theorem~1.9]{MR4937352}.
\end{remark}

For any Riemannian manifold $(M, g)$, we denote the Riemannian distance on $M$ by $d_g$. If $N \subset M$ is a submanifold, then we denote the restriction of $g$ to $N$ by $g_N$. The difference between the distance function $d_{g_N}$ and the restriction $d_g|_N$ of the Riemannian distance on $M$ to $N$ will play an important role throughout this paper.
\begin{theorem}\label{thm:bandwidthForBand}
	Let $(M^n,g)$ be a regular $\Spin$ Riemannian band with boundary $\partial_- M \sqcup \partial_+ M$. 
	Let $\ell$ denote the distance between $\partial_-M$ and $\partial_+M$ in $M$.
	Suppose that
	\begin{enumerate}
		\item the identity map
		\[
		\id\colon (\partial_-M,d_{g_{\partial_-M}})
		\to (\partial_-M,d_g|_{\partial_- M})
		\]
		is a coarse equivalence, and
		\item the Dirac operator $D_{\partial_-M}$ has non-zero index in $K_{n-1}(C^*(\partial_-M))$.
	\end{enumerate}
	Then
	\beq{eq band width ineq}
	\inf_{x\in M}\Sc_g(x)\le\frac{4\pi^2(n-1)}{n\ell^2}.
	\eeq
\end{theorem}
The first assumption in Theorem \ref{thm:bandwidthForBand} implies that $\partial_-M$ is connected: otherwise, the different connected components of $\partial_-M$ are at infinite distance to each other with respect to $d_{g_{\partial_-M}}$, but not with respect to $d_g$.

The bound in \eqref{eq band width ineq} is already sharp in the compact case, as pointed out below Conjecture \ref{con band width cpt}, so also in this more general setting.

\begin{example}
As in Theorem \ref{thm band width cpt}, let $X$ be a complete, $n$-dimensional Riemannian $\Spin$-manifold of dimension $n$, and let $N \subset X$ be a \emph{compact} hypersurface with trivial normal bundle. Suppose that the map $\pi_1(N) \to \pi_1(X)$ induced by the inclusion is injective. This implies that the universal cover $\tilde N$ of $N$ embeds into the universal cover $\tilde X$ of $X$. Let $M \subset X$ be a relatively compact neighborhood of $N$ of width $l$, whose boundary decomposes as $\partial M = \partial_-M \sqcup \partial_+M$, where $\partial_{\pm}M$ are compact hypersurfaces.  Then the universal cover $\tilde M$ of $M$ is a regular $\Spin$ Riemannian band. Suppose that $\bar M = M_- \cup M_+$, where $M_- \cap M_+ = N$,  $\partial(M_{\pm}) = \partial_{\pm}M \sqcup N$, and $M_-$ defines a cobordism between $\partial_-M$ and $N$.

Suppose that the coarse index of the Dirac operator $D_{\tilde N}$ on $\tilde N$ is nonzero. This index equals the coarse index of $D_{\partial \tilde M_-}$ by the cobordism defined by the universal cover $\tilde M_-$ of $\tilde M_-$ (see \cite{Wulff12}, \cite[Theorem 4.12]{Wulff19}, \cite[Corollary 2.27]{HochsdeKok25}), which is therefore also nonzero. 
The first condition in Theorem \ref{thm:bandwidthForBand} now holds automatically for $\tilde M$, by relative compactness of $M$.  This theorem therefore implies that 
\beq{eq inf Sc tilde M}
	\inf_{x\in M}\Sc_{g}(x) = \inf_{\tilde x\in \tilde M}\Sc_{\tilde g}(\tilde x)  \le\frac{4\pi^2(n-1)}{n\ell^2},
\eeq
where $\tilde g$ is the lift of $g$ to $\tilde X$.
In other words, Theorem \ref{thm:bandwidthForBand} reduces to a version of Theorem \ref{thm band width cpt} with nonvanishing of the Rosenberg index replaced by nonvanishing of the coarse  index of $D_{\tilde N}$.

For a result resembling Theorem \ref{thm band width cpt} more closely, 
suppose that the equivariant coarse index 
\beq{eq equivar coarse index}
\ind_{\pi_1(N)}(D_{\tilde N}) \in K_{n-1}(C^*(\tilde N)^{\pi_1(N)}) = K_{n-1}(C^*(\pi_1(N)))
\eeq
 is nonzero. This condition is slightly stronger than the nonvanishing of the Rosenberg index of $N$, which is an analogue of this index in $KO_{n-1}(C^*(\pi_1(N)))$.

A variation on the second condition in Theorem \ref{thm:bandwidthForBand} now holds: the equivariant index coarse of $D_{\partial_- \tilde M}$ equals \eqref{eq equivar coarse index} by the cobordism defined by $\tilde M_-$, and is hence nonzero.
Theorem \ref{thm:bandwidthForBand} immediately generalizes  to the case where $D_{\partial_-M}$ has non-zero index in $K_{n-1}(C^*(\partial_-M)^{\Gamma})$, for some discrete, finitely generated group $\Gamma$ acting properly and freely on $M$, preserving all structure. This version of the theorem implies \eqref{eq inf Sc tilde M}.
\end{example}

We prove Theorem \ref{thm:bandwidthForBand} in Section \ref{sec proof bandwidth}.  The main technical tool for this proof is a \emph{quantitative partitioned manifold index theorem} for possibly noncompact hypersurfaces, Theorem \ref{thm:quantitativeIndex}.

\section{A partitioned manifold index theorem for possibly noncompact hypersurfaces}

We state a variation of the main result from \cite{HochsdeKok25}, a generalization of Roe's partitioned manifold index theorem \cite{Higson91, Roe1988dualtoeplitz, Roe1996indextheory}, in which the partitioning hypersurface may be noncompact. Related results were obtained earlier in \cite{bunke2024coronas, Engel2025relative}. The result is Theorem \ref{thm:recallPartitionIndex}. Compared to Theorem 2.23 in \cite{HochsdeKok25}, the construction of the index is slightly different, and the result in this paper applies to both even- and odd-dimensional hypersurfaces. We point out why this result is \emph{not} enough to prove Theorem \ref{thm:bandwidthForBand}, which motivates the generalization of  Theorem \ref{thm:recallPartitionIndex} to Theorem \ref{thm:quantitativeIndex}. We give a direct proof of Theorem \ref{thm:recallPartitionIndex}  in a special case (Proposition \ref{prop:product}), which will be used in the proof of Theorem \ref{thm:quantitativeIndex} in Section \ref{sec quant index}. 

\subsection{Roe algebras and coarse maps}\label{sec Roe alg}

We now recall the definitions of Roe algebras and their localized variants. These algebras play a central role in index theory, as the indices of differential operators naturally take values in their $K$-theory.

\begin{definition}
	Let $X$ be a proper metric space (meaning that closed balls are compact). A Hilbert space $H$ is an \emph{$X$-module} if $H$ carries a nondegenerate representation of $C_0(X)$ (meaning that no nonzero vector in $H$ is mapped to zero by every element of $C_0(X)$). The module $H$ is called \emph{ample} if no nonzero function in $C_0(X)$ acts as a compact operator.
\end{definition}

For example, if $X$ is a complete Riemannian manifold, then $H = L^2(X,E)$ for a Hermitian vector bundle $E\to X$ gives an ample $X$-module.

\begin{definition}\label{def:propagation}
	Let $X$ be a proper metric space and $H$ an ample $X$-module.
	\begin{enumerate}
		\item An operator $T\in B(H)$ is \emph{locally compact} if $fT$ and $Tf$ are compact for all $f\in C_0(X)$.
		\item An operator $T\in B(H)$ has \emph{finite propagation} if there exists $P>0$ such that $fTg=0$ whenever $f,g\in C_c(X)$ satisfy $d(\mathrm{supp}(f),\mathrm{supp}(g))>P$.
        \item If $M\subset X$ is a closed subset, we say that $T\in B(H)$ is \emph{supported in $M$} if $fT=Tf=0$ for any $f\in C_c(X)$ supported away from $M$.
		\item If $N\subset X$ is a subset, we say that $T\in B(H)$ is \emph{supported near $N$} if there exists $R>0$ such that $T$ is supported in the closed $R$-neighborhood  $\bar B_R(N)$ of $N$. 
	\end{enumerate}
	The \emph{Roe algebra} $C^*(X)$ is the norm closure in $B(H)$ of the algebra of  locally compact operators with finite propagation.  
	If $N \subset X$, then the \emph{localized Roe algebra} $C^*(N\subset X)$ is the norm closure of all locally compact, finite-propagation operators supported near $N$.
\end{definition}
The Roe algebra $C^*(X)$ is independent of the particular choice of ample $X$-module; see \cite[Lemma~5.1.12]{higherindex}.

If $X$ is a complete $\Spin$-manifold of dimension $n$, then the $\Spin$-Dirac operator $D_X$ on $X$ has a \emph{coarse index}
\[
\ind(D_X) \in K_n(C^*(X)),
\]
as defined in \cite[Definition~3.7]{Roe1996indextheory} or \cite[Section~2.3.1]{XieYu14}.

A key property of Roe algebras is their invariance under coarse equivalence. We recall the basic notions.
\begin{definition}
	Let $(X, d_X)$ and $(Y, d_Y)$ be proper metric spaces, and let $f\colon X\to Y$ be a proper map.
	\begin{enumerate}
		\item The map $f$ is a \emph{coarse map} if there exists a nondecreasing function $\rho_+\colon \mathbb{R}_{\ge 0}\to \mathbb{R}_{\ge 0}$ with $\lim_{r\to \infty}\rho_+(r)=\infty$ such that
		\[
		d_Y(f(x_1),f(x_2)) \le \rho_+(d_X(x_1,x_2))
		\]
		for all $x_1,x_2\in X$.
		\item The map $f$ is a \emph{coarse embedding} if there exist nondecreasing functions $\rho_\pm\colon \mathbb{R}_{\ge 0}\to \mathbb{R}_{\ge 0}$ with $\lim_{r\to\infty}\rho_-(r)=\infty$ such that
		\[
		\rho_-(d_X(x_1,x_2)) \le d_Y(f(x_1),f(x_2)) \le \rho_+(d_X(x_1,x_2))
		\]
		for all $x_1,x_2\in X$.
		\item A coarse embedding $f$ is a \emph{coarse equivalence} if $f(X)$ is a \emph{net} in $Y$, i.e., there exists $R>0$ such that $Y = \bar B_R(f(X))$.
	\end{enumerate}
\end{definition}

\begin{theorem}[{\cite[Theorem~5.1.15]{higherindex}}]\label{thm:coarseMap}
	Let $X$ and $Y$ be proper metric spaces, and let $f\colon X\to Y$ be a coarse map. Then $f$ naturally induces a homomorphism
	\[
	f_*\colon K_*(C^*(X)) \longrightarrow K_*(C^*(Y)).
	\]
	This construction is functorial. 
	Moreover, if $f$ is a coarse equivalence, then $f_*$ is an isomorphism.
\end{theorem}
We will often use the special case of Theorem \ref{thm:coarseMap} where $X$ is a closed subset of $Y$, and $f$ is the inclusion map. To be more precise,
let $(Y, d_Y)$ be a proper metric space, and let $X\subset Y$ be a closed subset that equals the closure of its interior. Let $d_X$ be a distance function on $X$ that induces the same topology on $X$ as the restriction of $d_Y$, and such that for all $x,x' \in X$,
\[
d_Y(x,x') \leq d_X(x,x').
\]
It follows from this  inequality and closedness of $X$ that $(X, d_X)$ is a proper metric space. And the inclusion map $i_X\colon (X, d_X) \to (Y, d_Y)$ is continuous.

Let $H_Y$ be an ample $Y$-module. 
 The representation $C_0(Y) \to \cB(H_Y)$ extends to $L^{\infty}(Y)$ (with respect to a Borel measure), and in particular to the indicator function $1_X$ of $X$. We write $H_X = 1_X \cdot H_Y$. One can check directly, using the fact that $X$ equals the closure of its interior, that $H_X$ is an ample $X$-module.
 
Consider the Roe algebra $C^*(Y)$ defined on $H_Y$, and the Roe algebra $C^*(X)$ defined on $H_X$, with respect to the distance function $d_X$. Let $j\colon H_X \to H_Y$ be the inclusion map. Define the norm-continuous linear map $\iota_X\colon \cB(H_X) \to \cB(H_Y)$ by
\[
\iota_X(T) := j T j^*, 
\]
for $T \in \cB(H_X)$.
\begin{lemma}\label{lem funct map incl}
The map $\iota_X$ restricts to an injective $*$-homomorphism $\iota_X\colon C^*(X) \to C^*(Y)$, and the induced map on $K$-theory is the map $(i_X)_*$ from Theorem \ref{thm:coarseMap}.
\end{lemma}
\begin{proof}
The first claim follows by a straightforward verification. The second claim follows from the fact that $j$ is an isometry covering the inclusion $i_X$. (See e.g.\ \cite[Definitions 4.3.3 and 5.1.4]{higherindex}.)
\end{proof}
We will view $C^*(X)$ as a subalgebra of $C^*(Y)$ via the map $\iota_X$, and use the description of the map $(i_X)_*$ in Lemma \ref{lem funct map incl} without referring to this lemma explicitly. If an element of $C^*(Y)$ is supported in $X$ in the sense of point 3 in Definition \ref{def:propagation}, then it lies in the image of $\iota_X$, so we may view it as an element of $C^*(X)$.
We will say that an element  $a \in K_*(C^*(Y))$ is \emph{represented by} an element $b \in K_*(C^*(X))$ if the (possibly non-injective) map $(i_X)_*$  induced by $i_X$ on $K$-theory maps $b$ to $a$.

In cases where $N \subset Y$ is a closed subset, with a distance function $d_N$ such that identity map $\id\colon (N, d_N) \to (N, d_Y|_N)$ is a coarse equivalence, we will use an isomorphism $K_*(C^*(N)) \cong K_*(C^*(N \subset Y))$ in a few places. Let $X \subset Y$  be closed, and suppose that that $N$ lies in the interior of $X$, and that the inclusion  $i_N \colon N \to X$ is a coarse equivalence. Then we have isomorphisms
\beq{eq iso loc Roe}
K_*(C^*(N)) \xrightarrow[\cong]{(i_N)_*} K_*(C^*(X) \xrightarrow[\cong]{((i_{X})_*} K_*(C^*(N \subset Y)).
\eeq
This is Lemma 1 on page 93 of \cite{HRY93}, combined with Theorem \ref{thm:coarseMap} and Lemma \ref{lem funct map incl}. We also used that the image of the map $\iota_X$ in Lemma \ref{lem funct map incl} lies in $C^*(N \subset X)$ in this case.

\subsection{The partitioned manifold  index}\label{sec:3.1}

We first recall the definition of the partitioned manifold index for possibly noncompact hypersurfaces.

\begin{definition}
	Let $X^n$ be a manifold. An embedded codimension-one submanifold $N^{n-1}\subset X^n$ is called a \emph{partitioning hypersurface} if
	\[
	X = X_- \cup X_+,
	\]
	where $X_\pm$ are submanifolds with boundary satisfying
	\[
	N = X_+ \cap X_- = \partial X_\pm .
	\]
\end{definition}

\begin{definition}\label{def:definingFunction}
    Let $X^n$ be a manifold partitioned by $N$. A bounded Lipschitz function $\psi$ on $X$ is called a \emph{defining function} of the partition, if
    \begin{itemize}
        \item the range of $\psi$ is $[-A,A]$ for some $A>0$,
        \item there exists $R>0$ such that $\psi^{-1}((-A,A))\subset B_R(N)$,
        \item $\psi\equiv\pm A$ on $X_\pm\cap\left(X\backslash \psi^{-1}((-A,A))\right)$, and
        \item $N=\psi^{-1}(a)$ for some $a\in (-A,A)$.
    \end{itemize}
\end{definition}
By Theorem~\ref{thm:uniformTubular}, a defining function always exists when $N\subset X$ is a uniformly embedded hypersurface.

If a real vector bundle $E$ over a manifold  has a $\Spin$-structure, then we denote the associated spinor bundle by $S(E)$. 
Let $(X^n,g)$ be a complete $\Spin$ Riemannian manifold with bounded geometry, partitioned by a uniformly embedded hypersurface $N^{n-1}\subset X$. Let $\psi$ be a defining function of $N$. Denote by $D$ the Dirac operator on $X$ acting on  $S(TX)$.  
Then the pair $(D,\psi)$ determines an index class
\[
\ind(D,\psi)\ \in\ K_{n-1}\bigl(C^*(N\subset X)\bigr),
\]
in the ways described below.

\subsubsection*{The odd-dimensional case}

Assume first that $\dim X=n$ is odd. Define
\beq{eq B odd}
B
= 
\begin{pmatrix}
	0 & D\\
	D & 0
\end{pmatrix}
+
\begin{pmatrix}
	0 & -i\psi\\
	i\psi & 0
\end{pmatrix},
\qquad
F=\frac{B}{(1+B^2)^{1/2}}
=
\begin{pmatrix}
	0 & V\\
	U & 0
\end{pmatrix},
\eeq
where
\[
\begin{cases}
	U=(D+i\psi)\,(1+(D^2+\psi^2))^{-1/2},\\[0.3em]
	V=(D-i\psi)\,(1+(D^2+\psi^2))^{-1/2}.
\end{cases}
\]
We denote the multiplier algebra of a $C^*$-algebra $\mathcal{A}$ by $\mathcal{M}(\mathcal{A})$.
Then $F \in \mathcal{M}(C^*(N\subset X))$, and 
\beq{eq F2-1 odd}
F^2-1\ \in\ C^*(N\subset X),
\eeq
see \cite[Proposition~2.13]{HochsdeKok25}.  So this operator defines a class
\[
[F] \in K_1(\mathcal{M}(C^*(N \subset M))/C^*(N \subset M)).
\]
Define
\beq{eq part index odd}
\ind(D,\psi)
= \partial[F]
= [P]-
\left[\begin{pmatrix}
	1&0\\[0.2em]
	0&0
\end{pmatrix} \right]
\ \in\ K_0\bigl(C^*(N\subset X)\bigr),
\eeq
where
\[
P
=
\begin{pmatrix}
	1-(1-UV)^2 & (2-UV)\, U\,(1-VU)\\[0.2em]
	V(1-UV) & (1-VU)^2
\end{pmatrix}.
\]

\subsubsection*{The even-dimensional case}

Now assume that $\dim X=n$ is even.  
The spinor bundle $S(TX)$ carries a natural $\mathbb{Z}_2$-grading with grading operator $\gamma$.  
Define
\beq{eq B even}
B = D + \psi \gamma, 
\qquad 
F = \frac{B}{(1+B^2)^{1/2}}.
\eeq
We define
\beq{eq def p}
p=\frac{F+1}{2},
\eeq
which by \eqref{eq F2-1 odd} satisfies
\beq{eq p^2-p odd}
p^2-p\ \in\ C^*(N\subset X).
\eeq
So it defines
\[
[p] \in K_0(\mathcal{M}(C^*(N \subset M))/C^*(N \subset M)).
\]
Thus we define
\beq{eq part index even}
\ind(D,\psi)
= \partial[p]
= [\exp(2\pi i p)]
\in K_1\bigl(C^*(N\subset X)\bigr).
\eeq

\medskip


\subsection{A partitioned manifold index theorem}

The partitioned manifold index theorem is as follows.

\begin{theorem}\label{thm:recallPartitionIndex}
	Let $(X^n,g)$ be a complete $\Spin$ Riemannian manifold with bounded geometry, partitioned by a uniformly embedded hypersurface $N^{n-1}\subset X$.  
	Let $\psi\colon X\to[-A,A]$ be a defining function for the partition.  
	Let $D$ be the Dirac operator on $X$.  
	If
	\beq{eq cond N equiv}
	\id\colon (N,d_{g_N}) \longrightarrow (N,d_g|_N)
	\quad\text{is a coarse equivalence},
	\eeq
	then 
	\[
	\ind(D,\psi) = \ind(D_N) \in K_{n-1}(C^*(N)) \cong K_{n-1}(C^*(N \subset X)),
	\]
	where $\ind(D_N)$ denotes the coarse index of the Dirac operator $D_N$ on $N$ associated to the $\Spin$-structure induced by the one on $X$.
\end{theorem}

If $n$ is odd, then Theorem \ref{thm:recallPartitionIndex} is equivalent to Theorem~2.23 of \cite{HochsdeKok25} (in the case of $\Spin$-Dirac operators), although we use a slightly different representative for the index class.  The latter result is closely related to Theorem 7.5 in \cite{bunke2024coronas} and Corollary 7.6 in \cite{Engel2025relative}, which appeared earlier.

The main technical tool we develop in this paper is a generalization of Theorem \ref{thm:recallPartitionIndex}, namely Theorem~\ref{thm:quantitativeIndex} below. The point of this generalization is to weaken the assumption \eqref{eq cond N equiv} on the different distance functions on $N$. We will see that that condition is too strong to allow us to use Theorem \ref{thm:recallPartitionIndex} to prove Theorem \ref{thm:bandwidthForBand}, whereas Theorem~\ref{thm:quantitativeIndex} does apply. See Remark \ref{rem motivation quant index}.

We will obtain a proof of Theorem \ref{thm:recallPartitionIndex} as a special case of Theorem~\ref{thm:quantitativeIndex}.
For now, we recall the following key proposition, which proves Theorem~\ref{thm:recallPartitionIndex} in the model product case and will also be used in the proof of Theorem~\ref{thm:quantitativeIndex}.

\begin{proposition}\label{prop:product}
	Let $(N^{n-1},g_N)$ be a complete $\Spin$ Riemannian manifold with bounded geometry and consider $X = N\times \mathbb{R}$ equipped with the product metric $g = dt^2 + g_N$.  Consider the partitioning of $X$ by  $N = N\times\{0\}\subset X$, with $X_+ = N \times [0, \infty)$ and $X_- = N \times (-\infty, 0]$.
	Let $D_N$ be the Dirac operator on $N$.  
	Let $\psi$  be a defining function of $N$ that is constant in the factor $N$ in $N \times \R$, and let $\ind(D,\psi)$ denote the index class defined in \eqref{eq part index odd} or \eqref{eq part index even}.  
	Then
	\[
	\ind(D,\psi) = \ind(D_N) \in K_*(C^*(N)).
	\]
\end{proposition}
\begin{proof}
	The proof is based on $KK$-theory and follows the arguments of
	Section~2.3.3 in~\cite{XieYu14}.

	We first recall that for any stable $C^*$-algebra $\mathcal A$
	(i.e.\ $\mathcal A\otimes\mathcal K\cong\mathcal A$),
	the $KK$-groups are given by
	\[
	KK_0(\C,\mathcal A)
	=
	\Bigl\{\,[F]:
	F=
	\begin{pmatrix}
		0 & V \\
		U & 0
	\end{pmatrix}\in \mathcal M(\mathcal A),
	\ F=F^*,\ F^2-1\in \mathcal A
	\Bigr\},
	\]
	and
	\[
	KK_1(\C,\mathcal A)
	=
	\{\, [F]: F\in \mathcal M(\mathcal A),~F=F^*,\ F^2-1\in\mathcal A \,\};
	\]
 see~\cite[Section~17.5.4]{Blackadar}.
	Moreover, there is a natural isomorphism
	\begin{equation}\label{eq:KKiso}
		KK_*(\C,\mathcal A)
		\xrightarrow[~\cong~]{}
		K_{*+1}(\mathcal M(\mathcal A)/\mathcal A)
		\xrightarrow[~\cong~]{~\partial~}
		K_*(\mathcal A),
	\end{equation}
	where 
	 $\partial$ is the boundary map; see
	\cite[Propositions~17.5.5 and~17.5.6]{Blackadar}.

	We now take $\mathcal A=C^*(N\subset X)\cong C^*(N)$.
	By construction and by the definition of the boundary map
	$\partial$ (see \cite[Definitions~8.3.1 and~9.3.2]{Blackadar}),
	the index class
	\[
	\ind(D,\psi)\in K_{n-1}(C^*(N\subset X))
	\]
	is precisely the class
	$[F]\in KK_{n-1}(\C,C^*(N))$,
	where $F$ is the operator defined in Section~\ref{sec:3.1},
	for either parity of $n$. (The correspondence between unitary elements and projections just above \cite[Proposition~17.5.6]{Blackadar} is exactly as in \eqref{eq def p}.)

	Recall that the operator $B$ acts on $S(T(N\times\R))^{\oplus 2}$
	when $\dim N$ is even, and on $S(T(N\times\R))$
	when $\dim N$ is odd.
	In either case, the underlying bundle is naturally isomorphic to
	\[
	S(T(N\times\R)\oplus\hat \R)
	=
	S(TN\oplus\R\oplus\hat \R)
	\cong
	S(TN)\,\hat\otimes\, S(\R\oplus\hat \R),
	\]
	where $\hat \R$ denotes a trivial flat one-dimensional bundle
	with standard basis vector $\hat v$.
	Under this identification, the Clifford action $ic(\hat v)$
	corresponds to the matrix
	\(
	\begin{pmatrix}
		0 & -i \\
		i & 0
	\end{pmatrix}
	\)
	when $\dim N$ is even, and to the grading operator $\gamma$
	when $\dim N$ is odd.
	Consequently, the operator $B$ decomposes as the graded tensor sum
	\begin{equation}\label{eq:gradedTensor}
		B = D_N \,\hat\otimes\, 1 + 1 \,\hat\otimes\, B_\R,
	\end{equation}
	where $B_\R$ is the odd differential operator acting on
	$S(\R\oplus\hat\R)$ over $\R$, given by
	\[
	B_\R =
	\begin{pmatrix}
	0 & i\frac{d}{dt}\\  i\frac{d}{dt}&0
	 \end{pmatrix}
	 + \psi\cdot ic(\hat v).
	\]

	Consider the function $F(x)=\frac{x}{\sqrt{x^2+1}}$ and let $F(D_N)$ and $F(B_\R)$ denote the corresponding bounded transforms.
	These operators define classes
	\[
	[F(D_N)]\in KK_{n-1}(\C,C^*(N))
	\quad\text{and}\quad
	[F(B_\R)]\in KK_0(\C,\mathcal K),
	\]
	respectively.
	One standard approach to defining the Kasparov product
	$[F(D_N)]\otimes_{KK}[F(B_\R)]$
	is via connections, following Connes--Skandalis
	\cite{ConnesSkandalis}; see also
	\cite[Chapter~18]{Blackadar}.
	In particular, once a candidate Kasparov module is specified,
	it suffices to verify that it satisfies the conditions listed in
	\cite[Definitions~18.3.1 and~18.4.1]{Blackadar}.
	In the present situation, it is straightforward to check that
	the operator $F(B)$ defines such a module.
	Therefore,
	\[
	[F(B)] = [F(D_N)]\otimes_{KK}[F(B_\R)]
	\quad\text{in } KK_{n-1}(\C,C^*(N)).
	\]
	An alternative approach uses unbounded Kasparov modules;
	see~\cite{MR715325}.

	Finally, we observe that $B_\R$ is a Fredholm operator of index one (see Section 2 of \cite{Higson91}).
	Thus $[F(B_\R)]$ represents a generator of
	\[
	KK_0(\C,\mathcal K)\cong K_0(\mathcal K)\cong\Z.
	\]
	Moreover, for $\mathcal A=C^*(N)$, the image of $[F(D_N)]$
	under the isomorphism~\eqref{eq:KKiso}
	is precisely the index class
	\[
	\ind(D_N)\in K_*(C^*(N)).
	\]
	This completes the proof.
\end{proof}

\section{A quantitative partitioned manifold index}\label{sec quant index}

We develop a refinement of the partitioned manifold index: the \emph{quantitative partitioned manifold index}, see Definitions \ref{def quant index odd} and \ref{def quant index even}. In Section \ref{sec quant index thm}, we prove a quantitative partitioned manifold index theorem for this index.

\subsection{Estimates for bounded transforms}

In this subsection, we recall some estimates from \cite{WangJinmin25} that will be the technical basis of the construction of the quantitative partitioned manifold index in Definitions \ref{def quant index odd} and \ref{def quant index even}.

We begin by fixing some notation. Let $(X^n,g)$ be a complete $\Spin$ Riemannian manifold with bounded geometry, and let $N\subset X$ be a uniformly embedded hypersurface. 
\begin{definition}\label{def:Npsi}
	Given a defining function $\psi$ of the partition $N\subset X$, a \emph{$\psi$-neighborhood} of $N$ is a closed subset
	\[
	\mathcal N_\psi \subset \bar B_R(N)\subset X
	\]
	containing $\psi^{-1}((-A,A))$, such that for every $x\in \mathcal N_\psi$, there exists a geodesic $\gamma$ from $x$ to $N$ inside $\mathcal N_\psi$ with 
	$
	\mathrm{length}(\gamma)=d_g(x,N)\leq R.$
\end{definition}
Fix a $\psi$-neighborhood $\mathcal N_\psi$  of $N$ from now on. 
This always exists; for instance, one may take $\mathcal N_\psi=\bar B_R(N)$. We will also use more general choices, see e.g.\ \eqref{eq choice N psi}.

Let $\psi$ be as in Definition \ref{def:definingFunction}. Let $B$ be the operator \eqref{eq B odd} or \eqref{eq B even}, depending on the parity of $n$. 
Consider the operator
	\beq{eq def Fb}
	F_b=\frac{(bB)^2}{(1+b^2B^2)^{1/2}}.
	\eeq
	
    Choose a smooth  function $\chi\colon X\to [0,1]$ such that:
	\begin{itemize}
		\item $\chi$ is supported in $\bar B_{r}(\mathcal N_\psi)$;
		\item $\psi\equiv\pm A$ on $\operatorname{supp}(1-\chi)$;
		\item $\sqrt{\chi}$ is $2/r$–Lipschitz on $X$.
	\end{itemize}
    Here $r>0$ is to be chosen later.

We recall an estimate for the operator $B$. 
The following lemma is  \cite[Lemma 4.1]{WangJinmin25}, which is a version of \cite[ Lemma 2.11]{WXYZ24}. 
\begin{lemma}\label{lemma:bound}
Let $\rho$ be a smooth function with $0\le\rho\le1$, supported in $\{|\psi|=A\}$. Then for any $\lambda,b\in\R$,
	\[
	\bigl\|(1+\lambda^2+b^2B^2)^{-1/2}\rho\bigr\|\le \frac{1}{\sqrt{1+\lambda^2+b^2A^2}}.
	\]
\end{lemma}
%
%

	Using Lemma \ref{lemma:bound}, one obtains the following estimates, see  \cite[Lemma 4.2]{WangJinmin25}.
\begin{lemma}	\label{lem est Fb}
We have
\begin{align}
\label{eq1.1} &\|[F_b,\sqrt{\chi}]\| \leq 12\frac{\sqrt b}{r\sqrt A};\\
\label{eq1.2} &\|(F_b^2-1)(1-\chi)\|  \leq \frac{1}{bA} .
	\end{align}
\end{lemma}

\subsection{Construction of the index in odd dimensions}

Suppose  that the dimension $n$ of $X$ is odd.

Let $\mathcal N_\psi$ be as in Definition~\ref{def:definingFunction}. For any $r> 0$, consider the metric $d_r$ on $\bar B_r(\mathcal N_\psi)$ defined by
\[
d_r(x,y)\coloneqq \inf\left\{\int_\gamma |\dot\gamma(t)|_g\,dt:\ \gamma \text{ is a smooth path in } \bar B_r(\mathcal N_\psi) \text{ connecting } x \text{ and } y\right\}.
\]
The metric $d_r$ restricts to a metric on $N$, which we again denote by $d_r$. 
Clearly $d_r\le d_{g_N}$, so the identity map
\[
\id\colon (N,d_{g_N})\longrightarrow (N,d_r)
\]
is a coarse map.

Write
	\[
	F_b=\frac{(bB)^2}{(1+b^2B^2)^{1/2}}
	=\begin{pmatrix}
		0&V_b\\ U_b&0
	\end{pmatrix},
	\]
	where
	\[
	\begin{cases}
		U_b=b(D+i\psi)\,(1+b^2(D^2+\psi^2))^{-1/2},\\[0.3em]
		V_b=b(D-i\psi)\,(1+b^2(D^2+\psi^2))^{-1/2}.
	\end{cases}
	\]
Now define
\beq{eq def tilde P}
\tilde P=\sqrt{\chi} P_b\sqrt{\chi}+\begin{pmatrix}
    1&0\\0&0
\end{pmatrix}(1-\chi),
\eeq
where
\beq{eq def Pb}
P_b=\begin{pmatrix}
    1-(1-U_bV_b)^2&(2-U_bV_b)U_b(1-V_bU_b)\\V_b(1-U_bV_b)&(1-V_bU_b)^2
\end{pmatrix}.
\eeq
\begin{proposition}\label{prop quasi proj}
We have $\tilde P\in M_2(C^*(\bar B_{r}(\mathcal N_\psi))^+)$, and $\tilde P-\begin{pmatrix}
    1&0\\0&0
\end{pmatrix}\in M_2(C^*(\bar B_{r}(\mathcal N_\psi))$.
By choosing 
$$b \geq\frac{1}{A\delta} \quad \text{and} \quad r\geq \sqrt{\frac{b}{A\delta}},$$
for some universal small constant $\delta>0$, we have
\[
\|\tilde P\|\leq 3 \quad \textup{and} \quad \|\tilde P^2-\tilde P\|<\frac{1}{1000}.
\]
\end{proposition}
\begin{proof}
We first note that
\begin{equation}\label{eq:1000-P}
    \tilde P=\begin{pmatrix}
    1&0\\0&0
\end{pmatrix}+\sqrt{\chi}\begin{pmatrix}
    -(1-U_bV_b)^2&(2-U_bV_b)U_b(1-V_bU_b)\\V_b(1-U_bV_b)&(1-V_bU_b)^2
\end{pmatrix}\sqrt\chi.
\end{equation}
Each entry of 
	\(
	P_b - 
	\begin{pmatrix}1&0\\0&0\end{pmatrix}
	\)
equals a matrix component of $\varphi(B_b)$, for some $\varphi \in C_0(X)$, which lies in $C^*(X)$ by \cite[Proposition 3.6]{Roe1996indextheory} or  \cite[Proposition 10.5.6]{Higson00}. Thus the first sentence of the proposition follows.

	We have
\beq{eq P tilde Pb}
	\tilde P
=
		\sqrt{\chi}[P_b,\sqrt{\chi}] + P_b 
		+ (1-\chi)\left(
		\begin{pmatrix}1&0\\0&0\end{pmatrix}-P_b
		\right).
\eeq	
The first term on the right  is smaller than or equal to a universal constant times $12\frac{\sqrt b}{r\sqrt A}\leq 12\sqrt{\delta}$ by \eqref{eq1.1}.	
By \eqref{eq1.2}, the last term in \eqref{eq P tilde Pb} is smaller than or equal to a universal constant times $1/(bA)=\delta$. So we can choose $\delta$ small enough such that $\|\tilde P - P_b\|< 1/6000$. 

Using the fact that $F_b$ and $F_b^2 - 1$ have operator norms at most $1$, one finds that for all $b>0$, 
$
\|P_b\| \leq 2.
$
So, for any choice of $b$ and $r$ as above, we have $\|\tilde P\| \leq 3$. Furthermore, using the fact that $P_b^2 = P_b$ to write
\[
\tilde P^2 - \tilde P = (\tilde P - P_b)\tilde P + P_b(\tilde P - P_b) + (P_b - \tilde P),
\]
we see that
\[
\|\tilde P^2 - \tilde P\| \leq \|\tilde P - P_b \|(\|\tilde P\| + \|P_b\|+ 1) <\frac{1}{1000}.
\]
Thus we obtain the desired bounds.
\end{proof}

By Proposition \ref{prop quasi proj}, the spectrum of $\tilde P$ excludes the line $\{z:\Real(z)=1/2\}$. So it lies inside the domain of the holomorphic function $\Theta$ on $\{z:\Real(z)\ne 1/2\}$ defined by	\[
	\Theta(z)=
	\begin{cases}
		0,& \Real(z)<1/2,\\
		1,& \Real(z)>1/2.
	\end{cases}
	\]
The operator $\Theta(\tilde P)$ defined by functional calculus is a genuine idempotent that is close to $\tilde P$. It lies in $M_2(C^*(\bar B_{r}(\mathcal N_\psi))^+)$ because this algebra, like every $C^*$-algebra, is closed under holomorphic functional calculus. 
\begin{lemma}\label{lemma:Theta-P}
Let $\varepsilon  \in (0,1/4)$.
If $P$ is a bounded operator on a Hilbert space such that $\|P^2-P\|<\varepsilon$, then
\[
 \|\Theta(P)-P\|< \frac{8\varepsilon}{1-4\varepsilon}(\|P\|+ 2\varepsilon).
\]
\end{lemma}
\begin{proof}
If $\lambda \in \Spec(P)$, then 
$
|\lambda| \cdot |\lambda - 1| = |\lambda^2-\lambda|< \varepsilon. 
$
Hence 
$|\lambda|<2\varepsilon$ (if $\Real(\lambda) < 1/2$), or
$|\lambda-1|<2\varepsilon$ (if $\Real(\lambda) > 1/2$). So
\[
P = \frac{1}{2\pi i}\int_{|\xi|=2\varepsilon}\xi(\xi-P)^{-1}d\xi + \frac{1}{2\pi i}\int_{|\xi-1|=2\varepsilon}\xi(\xi-P)^{-1}d\xi,
\]
and 
    $$\Theta(P)=\frac{1}{2\pi i}\int_{|\xi-1|=2\varepsilon}(\xi-P)^{-1}d\xi.$$
    Therefore,
   \beq{eq expr Theta P} 
    \Theta(P)-P=\frac{1}{2\pi i}\int_{|\xi-1|=2\varepsilon}(1-\xi)(\xi-P)^{-1}d\xi-\frac{1}{2\pi i}\int_{|\xi|=2\varepsilon}\xi(\xi-P)^{-1}d\xi.
    \eeq
    Hence it suffices to estimate $\|(\xi-P)^{-1}\|$. 
    
    For any $\xi\in\C$ with $|\xi-1|=2\varepsilon$, set $\eta=1-\xi$ with $|\eta|=2\varepsilon$ and
    $$w=\xi-\xi^2=\eta-\eta^2$$
    with $2\varepsilon(1-2\varepsilon)\leq |w|\leq 2\varepsilon(1+2\varepsilon)$. And
    $$(\xi-P)^{-1}=(\eta-P)(w-(P-P^2))^{-1}.$$
We estimate
\[
\begin{split}
   \| (w-(P-P^2))^{-1}\| &=\left\| w^{-1}\sum_{k=0}^\infty\left(w^{-1}(P-P^2)\right)^k\right\|\\
   &\leq  |w|^{-1}\sum_{k=0}^\infty \|w^{-1}(P-P^2)\|^k\\
   &=\frac{1}{|w|-\|P-P^2\|}\\
   &< \frac{1}{\varepsilon(1-4\varepsilon)}.
\end{split}
\]
So
\[
\|(\xi-P)^{-1}\|< \frac{\|P\|+2\varepsilon}{\varepsilon(1-4\varepsilon)}.
\]
    We obtain
\[
        \left\|\frac{1}{2\pi i}\int_{|\xi-1|=2\varepsilon}(1-\xi)(\xi-P)^{-1}d\xi\right\|
        \leq  \frac{1}{2\pi}\cdot 4\pi \varepsilon \cdot 2\varepsilon\cdot \frac{\|P\|+2\varepsilon}{\varepsilon(1-4\varepsilon)}=
        \frac{4\varepsilon}{1-4\varepsilon}(\|P\|+2\varepsilon).
\]
The norm of the other term on the right hand side of \eqref{eq expr Theta P} can be estimated similarly.
\end{proof}

By Definition~\ref{def:Npsi}, for any point $x\in \bar B_r(\mathcal N_\psi)$, there exists a smooth path lying in $\bar B_r(\mathcal N_\psi)$ connecting $x$ to $N$ with length at most $r+R$. Thus $N\subset \bar B_R(N)$ is a net. Consequently, the inclusion map
\beq{eq N cN equiv}
i_N\colon (N,d_r)\longrightarrow (\bar B_r(\mathcal N_\psi),d_r)
\eeq
is a coarse equivalence, which induces an isomorphism on $K$-theory of Roe algebras by Theorem \ref{thm:coarseMap}.
\begin{definition}[Quantitative partitioned manifold index for odd $n$]\label{def quant index odd}
Let $\delta,r, b>0$ be as in Proposition \ref{prop quasi proj}. 
The \emph{quantitative partitioned manifold index} of $D$ with respect to $\psi$ is 
	\beq{eq q ind odd}
	\ind_q(D,\psi)
	=
	[\Theta(\tilde P)]
	-
	\left[
	\begin{pmatrix}
		1&0\\0&0
	\end{pmatrix}
	\right]
	\in 
	K_0(C^*(\bar B_{r}(\mathcal N_\psi)))		\cong K_{0}(C^*(N,d_{r})) ,
		\eeq
		where $C^*(N,d_{r})$ denotes the Roe algebra of $(N,d_{r})$, and we have used the coarse equivalence \eqref{eq N cN equiv}. 
\end{definition}

\subsection{Construction of the  index in even dimensions}

Now assume that the dimension $n$ of $X$ is even.

\begin{lemma}\label{lem exp series}
There is a sequence $(a_k)_{k=0}^{\infty
}$ in $\C$ such that $a_k z^k$ is bounded in $k$ for all $z \in \C$, and such that for every bounded, self-adjoint operator $T$ on a Hilbert space,
\beq{eq exp series}
\exp(2\pi i T)-1 = (T^2-T)\sum_{k=0}^{\infty} a_k T^k.
\eeq
\end{lemma}
\begin{proof}
The function $z \mapsto e^{2\pi i z} - 1$ is entire, with zeroes at the integers. Hence $z\mapsto \frac{e^{2\pi i z}-1}{z(z-1)}$ is an entire function, and is therefore given by a power series that converges on all of $\C$. The coefficients $a_n$ in this power series have the desired property.
\end{proof}

Let $\delta, b,r>0$.
Let 
$\chi$ be as above Lemma \ref{lemma:bound}, and
\[
 B = D + \psi \gamma, 
\qquad 
F_b = \frac{bB}{(1+(bB)^2)^{1/2}}.
\]
as in \eqref{eq B even}.
Define
$$p_b=\frac{F_b+1}{2},\quad \textup{and} \quad Q_b=\exp\left(2\pi i p_b\right)\in C^*(N\subset X)^+$$
as in Section \ref{sec:3.1}. 
%
%
Consider the operator
\beq{eq def tilde Q}
\tilde Q=\sqrt{\chi} Q_b\sqrt{\chi}+(1-\chi)=1+\sqrt{\chi}(Q_b-1)\sqrt\chi.
\eeq
Because $Q_b-1 = \varphi(bB)$, for a function $\varphi \in C_0(\R)$, we have $Q_b -1\in C^*(X)$ by  \cite[Proposition 3.6]{Roe1996indextheory}, and hence $\tilde Q \in C^*(\bar B_{r}(\mathcal N_\psi))^+$.
\begin{proposition}\label{prop tilde Q Qb}
For a universal $\delta$ small enough, $b \geq\frac{1}{A\delta}$ and  $r\geq \sqrt{\frac{b}{A\delta}}$, we have 
$$\|\tilde Q-Q_b\|\leq \frac{1}{100}.$$
\end{proposition}
\begin{proof}
We have
\beq{eq tilde Q Qb}
\tilde Q-Q_b = \sqrt{\chi}[Q_b, \sqrt{\chi}]+ (1-\chi)(1-Q_b). 
\eeq
And
\[
\begin{split}
[Q_b, \sqrt{\chi}]&= \sum_{j=0}^{\infty}\frac{(2\pi i)^j}{j! 2^j}[(F_b+1)^j, \sqrt{\chi}].
\end{split}
\]
The norm of the right hand side is less than or equal to
\[
\left(
\sum_{j=0}^{\infty}\frac{\pi^j}{(j-1)!}  \right)   \|F_b+1 \| \cdot  \| [F_b, \sqrt{\chi}] \|.
\]
By \eqref{eq1.1}, this can be made arbitrarily small for small enough $\delta$.

For the second term on the right hand side of \eqref{eq tilde Q Qb}, we use the numbers $a_k$ from Lemma \ref{lem exp series} to write
\beq{eq Qb - 1}
Q_b-1 = (p_b^2-p_b)\sum_{k=0}^\infty a_k p_b^k=\frac 1 4(F_b^2-1)\sum_{k=0}^\infty a_k \left(\frac{F_b+1}{2}\right)^k.
\eeq
The norm of the series on the right hand side is bounded uniformly in $b$, since $\|F_b\|\leq 1$. So by \eqref{eq1.2}, we see that the norm of the second term on the right hand side of \eqref{eq tilde Q Qb} can be made arbitrarily small, if we choose $\delta$ small enough.
\end{proof}
The operator $Q_b$ is unitary, so
in the setting of Proposition \ref{prop tilde Q Qb}, the operator $\tilde Q$ is invertible in $C^*(\bar B_{r}(\mathcal N_\psi))^+$.
\begin{definition}[Quantitative partitioned manifold index for even $n$]\label{def quant index even}
Let $\delta, b, r>0$ be as in Proposition \ref{prop tilde Q Qb}. 
The \emph{quantitative partitioned manifold index} of $D$ with respect to $\psi$ is 
	\beq{eq q ind even}
	\ind_q(D,\psi)
	=
	[\tilde Q]\in
	K_1(C^*(\bar B_{r}(\mathcal N_\psi)))		\cong K_{1}(C^*(N,d_{r})).
		\eeq
\end{definition}

\begin{remark}\label{rem O A-1}
In Propositions \ref{prop quasi proj} and  \ref{prop tilde Q Qb}, and hence in Definitions \ref{def quant index odd} and \ref{def quant index even}, a  natural choice is $b=r=\frac{1}{A\delta}$. Then one obtains
\[
\ind_q(D, \psi) \in K_{n-1}(C^*(\bar B_{(A\delta)^{-1}}  (\mathcal{N}_{\psi}))).
\]
In Lemma \ref{lem part 2}, we make a different choice.
We discuss the dependence of the index on $b$ and $r$ in Lemma \ref{lem ind indep choices} below.
\end{remark}

\subsection{Basic properties}

We discuss some basic properties of the quantitative partitioned manifold index of Definitions \ref{def quant index odd} and \ref{def quant index even}. These are: independence of choices made (Lemma \ref{lem ind indep choices}), a relation with the partitioned manifold index from Subsection \ref{sec:3.1} (Lemma \ref{lem part 1}), and a vanishing result (Lemma \ref{lem part 2}).

\begin{lemma}\label{lem ind indep choices}
\begin{itemize}
\item[(a)] Given $\delta$ and $r$ as in Proposition \ref{prop quasi proj} or  \ref{prop tilde Q Qb}, the index $\ind_q(D,\psi)$ does not depend on $b$ as in  Proposition \ref{prop quasi proj} or  \ref{prop tilde Q Qb}.
\item[(b)] Given $\delta$, $b$ and $r$ as in Proposition \ref{prop quasi proj} or  \ref{prop tilde Q Qb}, the  index $\ind_q(D,\psi)$  does not depend on the function $\chi$ as above Lemma \ref{lemma:bound}.
\item[(c)]
 Let $\delta$ and $b$ be as in Proposition \ref{prop quasi proj} or  \ref{prop tilde Q Qb}, and let $r_1\geq r_0\geq \sqrt{\frac{b}{A\delta}}$. 
 Denote the corresponding quantitative indices by
\[
	\ind_q^{r_j}(D,\psi)
	\in 
	K_{n-1}(C^*(\bar B_{r_j}(\mathcal N_\psi))).
\]
With respect to the inclusion map $i_{\bar B_{r_0}(\mathcal N_\psi)}\colon \bar B_{r_0}(\mathcal N_\psi) \to \bar B_{r_1}(\mathcal N_\psi)$, we have
\beq{eq q ind dep r}
(i_{\bar B_{r_0}(\mathcal N_\psi)})_*(\ind_q^{r_0}(D,\psi)) = \ind_q^{r_1}(D,\psi).
\eeq
\end{itemize}
\end{lemma}
\begin{proof}
Part (a) follows by a simple homotopy.

For part (b), suppose that $\chi_0$ and $\chi_1$ are  functions as above  Lemma \ref{lemma:bound}.
For $t \in [0,1]$, 
consider the function
\beq{eq def chi t}
\chi_t :=( (1-t)\sqrt{\chi_0} + t \chi_1)^2.
\eeq
Then $\chi_t$ has the properties above  Lemma \ref{lemma:bound} for all $t \in [0,1]$.
Indeed,  the first and third properties holds for all $t$ by direct verifications. The second property holds, because
\[
\supp(1-\chi_t) \subset \supp(1-\chi_0) \cup \supp(1-\chi_1)
\]
for all $t$. 

The quantitative indices defined with respect to the functions $\chi_t$ define a homotopy between the the quantitative indices defined with respect to $\chi_0$ and $\chi_1$, so (b) follows.

For part (c),  let $\chi_j$  be a function as above  Lemma \ref{lemma:bound}, for $j =0,1$ and $r = r_j$. The function $\chi_t$ in \eqref{eq def chi t} satisfies the first condition above Lemma \ref{lemma:bound} with $r = r_1$, the second condition (which does not involve $r$), and the third condition for $r = r_0$. 

First, suppose that $n$ is odd.
For $t \in [0,1]$, let $\tilde P_{t}$ be the projection \eqref{eq def tilde P}, defined with $\chi =\chi_t$. The proof of Proposition \ref{prop quasi proj} shows that the conclusions of this result hold for $\tilde P_t$, for all $t \in [0,1]$, for $r = r_1$. So we have corresponding classes
\[
	[\Theta(\tilde P_t)]
	-
	\left[
	\begin{pmatrix}
		1&0\\0&0
	\end{pmatrix}
	\right]
	\in 
	K_0(C^*(\bar B_{r_1}(\mathcal N_\psi))).	
\]
(These are not quantitative indices as defined in Definition \ref{def quant index odd}, since the function $\chi_t$ does not necessarily satisfy the third condition above Lemma \ref{lemma:bound} for $r = r_1$.) These classes define a homotopy between the two sides of \eqref{eq q ind dep r}.

Similarly, if $n$ is even, then we define $\tilde Q_t$ as the operator \eqref{eq def tilde Q},  with $\chi = \chi_t$.  By the properties of $\chi_t$ mentioned above, 
the proof of Proposition \ref{prop tilde Q Qb} shows that the conclusions of Proposition \ref{prop tilde Q Qb}  hold for $\tilde Q_t$, with $r = r_1$. Hence these operators define a homotopy between the two sides of \eqref{eq q ind dep r}.
\end{proof}
%
%
%

\begin{remark}
The parameter $r$ in Definitions \ref{def quant index odd} and \ref{def quant index even} is not included in the notation $\ind_q$, but it plays an important role in the quantitative partitioned manifold index theorem, Theorem \ref{thm:quantitativeIndex} and its applications. By part (c) of Lemma \ref{lem ind indep choices}, the dependence of the quantitative index on $r$ is natural (the index for a given $r$ \emph{represents} the index for $r'>r$ in the sense below Lemma \ref{lem funct map incl}).
\end{remark}

Recall from Subsection \ref{sec Roe alg} that the image of the map from $C^*(\bar B_{r}(\mathcal N_\psi))$ to $C^*(X)$ induced by the inclusion $i_{\bar B_{r}(\mathcal N_\psi)}:\bar B_{r}(\mathcal N_\psi)\to X$ lies in $C^*(N \subset X)$.
\begin{lemma} \label{lem part 1}
The index from Definitions \ref{def quant index odd} and  \ref{def quant index even}  satisfies
		\[
		(i_{\bar B_{r}(\mathcal N_\psi)})_*(\ind_q(D,\psi))=\ind(D,\psi).
		\]
\end{lemma} 
\begin{proof}
We first assume that $n$ is odd.
As in the proof of Proposition~\ref{prop quasi proj}, for the chosen parameters $b$ and $r$ we have
\[
\|\tilde P-P_b\|<\frac{1}{1000}
\quad\text{in } M_2(C^*(N\subset X)^+).
\]
Lemma \ref{lemma:Theta-P} therefore implies that $\|\Theta(\tilde P) - \Theta(P_p)\| \leq \frac{3}{100}$, so the  idempotents $\Theta(\tilde P)$ and $\Theta(P_p)$ are homotopic. Hence 
\begin{multline*} [\Theta(\tilde P)]- \left[ \begin{pmatrix} 1&0\\0&0 \end{pmatrix} \right] = [\Theta(P_b)]- \left[ \begin{pmatrix} 1&0\\0&0 \end{pmatrix} \right]\\ = [P_b]- \left[ \begin{pmatrix} 1&0\\0&0 \end{pmatrix} \right] = \ind(D,\psi) \in K_0(C^*(N\subset X)). \end{multline*}
This proves the first point of Theorem~\ref{thm:quantitativeIndex} when $n$ is odd. 

The case where $n$ is even follows similarly from the estimate
\[
\|\tilde Q-Q_b\|<\frac{1}{100},
\]
established in Proposition \ref{prop tilde Q Qb}.
\end{proof}

\begin{lemma} \label{lem part 2}
There are constants $a, c_1>0$ such that, if the operators $D\pm i\psi$ are invertible on $X$, with
		\beq{eq D psi pos}
		(D\pm i\psi)^*(D\pm i\psi)\ge L^2
		\eeq
		for some $L>0$, and $A\geq a L$, then
		\[
		\ind_q(D,\psi)=0\in
		K_{n-1}\!\left(C^*\bigl(\bar B_{c_1(AL)^{-1/2}}(\mathcal N_\psi)\bigr)\right)
		\cong K_{n-1}(C^*(N,d_{c_1(AL)^{-1/2}})).
		\]		
\end{lemma} 
\begin{proof}
Let $\delta$ be as in Proposition \ref{prop quasi proj} or  \ref{prop tilde Q Qb}, depending on the parity of $n$. 
Assume that~\eqref{eq D psi pos} holds.
Then the spectrum of $B$ is contained in $(-\infty,-L]\cup[L,\infty)$.
Choose $b_1=\frac{1}{\delta_1 L}$, 
with $\delta_1 \in (0, \delta)$ sufficiently small, as specified below, and let $F_{b_1}$ be as in \eqref{eq def Fb}. Then
\beq{eq F2 - 1 invtble}
\|F_{b_1}^2-1\|
=
\left\|
\frac{1}{1+b_1^2B^2}
\right\|
\le
\frac{1}{1+b_1^2L^2}
<
\delta_1^2.
\eeq

Next, set $r_1=\sqrt{\frac{b_1}{A\delta_1}}$, and choose a cutoff function $\chi_1$ as above Lemma \ref{lemma:bound}, for $r=r_1$.
By Lemma~\ref{lem est Fb}, we obtain
\[
\|[F_{b_1},\sqrt{\chi_1}]\|
\le 
12\frac{\sqrt {b_1}}{r_1\sqrt A}
=
12\sqrt{\delta_1}.
\]

If $n$ is odd, then by~\eqref{eq:1000-P} and \eqref{eq F2 - 1 invtble}, the operator $\tilde P_1$ associated to these data as in \eqref{eq def tilde P} satisfies
\[
\left\|
\tilde P_1-
\begin{pmatrix}
	1&0\\
	0&0
\end{pmatrix}
\right\|
<
\frac{1}{100},
\]
provided $\delta_1$ is chosen sufficiently small. So
\beq{eq class P1}
	[\Theta(\tilde P_1)]
	-
	\left[
	\begin{pmatrix}
		1&0\\0&0
	\end{pmatrix}
	\right] = 0 \quad
	\in 
	K_0(C^*(\bar B_{r_1}(\mathcal N_\psi))).	
\eeq

If $n$ is even, an analogous argument shows that the operator $\tilde Q_1$ as in \eqref{eq def tilde Q} satisfies
\[
\|\tilde Q_1-1\|<\frac{1}{100}.
\]
So
\beq{eq class Q1}
	[\tilde Q_1] = 0 \quad
	\in 
	K_1(C^*(\bar B_{r_1}(\mathcal N_\psi))).	
\eeq

Now let $a = \delta_1/\delta$, and suppose that $A \geq aL$. Then $b_1 \geq \frac{1}{A\delta}$, and, since $\delta_1 < \delta$,
\[
r_1 \geq \sqrt{\frac{b_1}{A\delta}}.
\]
So $b = b_1$ and $r= r_1$ satisfy the conditions of Propositions \ref{prop quasi proj} and  \ref{prop tilde Q Qb}, and we have
\[
\ind_q(D, \psi) \in K_{*}(C^*(\bar B_{r_1}(\mathcal N_\psi))).
\]
This index is exactly the class \eqref{eq class P1} or \eqref{eq class Q1}, and hence zero.

Finally, note that $r_1 = c_1 (AL)^{-1/2}$, with $c_1 = 1/\delta_1$.
\end{proof}

\section{The quantitative partitioned manifold index theorem}\label{sec quant index thm}

In this section, we state and  prove a quantitative version of the partitioned manifold index theorem for noncompact manifolds, Theorem \ref{thm:quantitativeIndex}, refining Theorem \ref{thm:recallPartitionIndex}. This is the key step in our proof of Theorem \ref{thm:bandwidthForBand}.

\subsection{Statement of the result}

\begin{theorem}[Quantitative partitioned manifold index theorem]\label{thm:quantitativeIndex}
	Let $(X^n,g)$ be a complete $\Spin$ manifold with bounded geometry, partitioned by a uniformly embedded hypersurface $N^{n-1}\subset X$. Let $\psi\colon X\to[-A,A]$ be a defining function of $N\subset X$ such that
	\[
	\psi^{-1}((-A,A))\subset\mathcal N_\psi\subset \bar B_R(N).
	\]
	Let $D$ be the Dirac operator on $X$. Let $r>0$ be as in Proposition \ref{prop quasi proj} or  \ref{prop tilde Q Qb}, depending on the parity of $n$.
	There exists a universal constant $c_2>0$ such that,
		 with respect to the coarse map
		\[
		\id\colon (N,d_{g_N})\longrightarrow (N,d_{c_2A^{-1}}),
		\]
	we have 
		\[
		\ind_q(D,\psi)= \id_*(\ind(D_N))\in K_{n-1}(C^*(N, d_{c_2 A^{-1}})),
		\]
		where $\ind(D_N)$ is the coarse index of the Dirac operator on $N$.
\end{theorem}

\begin{remark}
In the setting of 
Theorem \ref{thm:recallPartitionIndex}, the combination of Lemma \ref{lem part 1} and Theorem \ref{thm:quantitativeIndex} implies that
\[
(i_{\bar B_{c_2A^{-1}}(\mathcal N_\psi)})_*(\ind(D_N)) = \ind(D, \psi).
\]
Because the composition
\[
K_*(C^*(N)) \xrightarrow{\cong} K_*(\bar B_{c_2 A^{-1}}(\mathcal{N}_{\psi})) \xrightarrow[\cong]{(i_{\bar B_{c_2A^{-1}}(\mathcal N_\psi)})_*} K_*(C^*(N \subset X))
\]
is the isomorphism \eqref{eq iso loc Roe} used in Theorem \ref{thm:recallPartitionIndex}, it follows that Theorem \ref{thm:recallPartitionIndex} is a special case of Theorem \ref{thm:quantitativeIndex}. 
\end{remark}

\begin{remark}
In this paper, and in particular in Theorem \ref{thm:quantitativeIndex}, we consider $\Spin$-Dirac operators. However, the construction of the quantitative partitioned manifold index, and Theorem \ref{thm:quantitativeIndex}, extend to more general Dirac-type operators. In Theorem \ref{thm:quantitativeIndex}, one then needs to specify a relation between $D$ and $D_N$, as in  \cite[Theorem 1.5]{Higson91}. We only discuss the case of $\Spin$-Dirac operators here, as this is what is needed to prove Theorem \ref{thm:bandwidthForBand}.
\end{remark}

Let us explain the difference between Theorems \ref{thm:recallPartitionIndex} and  \ref{thm:quantitativeIndex} by the following example.

\begin{example}\label{ex quant index thm}
	Let $Y=\R\times\R^{n-1}$ be the flat Euclidean space with $n\geq 3$, and let $\delta>0$ be a small number. Let $\mathcal L$ be the lattice $\Z^{n-1}$ in $\R^{n-1}$, and let $S\colon \mathcal L\to\mathcal L$ denote the reflection across the hyperplane orthogonal to the $x_1$-axis. Denote by $\mathcal L^+$ the collection of points in $\mathcal L$ with positive $x_1$-coordinate. 
	
	In $[0, \infty) \times \R^{n-1}$, remove the sets $B_\delta([1,+\infty)\times\{q\})$ and $B_\delta([1,+\infty)\times\{S(q)\})$ for each $q\in\mathcal L^+$. Attach to each removed region a cylinder $T_q\cong \partial B_\delta([1,+\infty)\times\{q\})\times[0,1]$. In this way we obtain a Riemannian manifold $(X,g)$. More precisely,
	\[
	X=\left(Y\setminus\bigcup_{q\in\mathcal L^+}\bigl(B_\delta([1,+\infty)\times\{q\})\cup B_\delta([1,+\infty)\times\{S(q)\})\bigr)\right)
	\cup\left(\bigcup_{q\in\mathcal L^+} T_q\right).
	\]
	We assume that all the $T_q$ are isometric, and that $g$ is a product metric on $\partial B_\delta([2,+\infty)\times\{q\})\times[0,1] \subset T_q$. 
	
	Note that $(X,g)$ has bounded geometry. Let $\rho$ be a smooth modification of the signed distance function to the slice $\{0\}\times\R^{n-1}$; in particular, $\rho$ is essentially the projection onto the first factor $\R$ in $Y$. 
	
	Let $f$ be a $2$-Lipschitz, smooth, non-decreasing function on $\R$ such that $f(x)\equiv \pm1$ for $|x|\geq 1$. Define $\psi_a=f \circ (\rho-a)$, a family of smooth functions on $X$ parametrized by $a \in \R$. As in Section \ref{sec:3.1} and \cite{HochsdeKok25}, the index class $\ind(D,\psi_a)\in K_*(C^*(\rho^{-1}(a)\subset X))$ is independent of $a$, and is equal to zero. Indeed, take $a=10$. Then $\rho^{-1}(10)$ is obtained from $\R^{n-1}$ by removing the $\delta$-balls $B_\delta(q)$ and $B_\delta(S(q))$ for all $q\in\mathcal L^+$ and attaching cylinders $T_q'\cong S^{n-2}\times[0,1]$. In this case $N=\rho^{-1}(10)$ satisfies the hypotheses of Theorem \ref{thm:recallPartitionIndex}; in particular, the identity map
	\[
	\id\colon (N,d_{g_N})\longrightarrow (N,d_g|_N)
	\]
	is a coarse equivalence. Therefore $\ind(D,\psi_a)\in K_*(C^*(N\subset X))$ vanishes by Theorem \ref{thm:recallPartitionIndex}, since
	\[
	K_*(C^*(N))\cong K_*(C^*(\mathcal L^+))=0.
	\]
	
	However, by Lemma \ref{lem part 1}, $\ind(D,\psi_a)$ admits a preimage 
	\[
	\ind_q(D,\psi_a)\in K_*\!\left(C^*([a-c,a+c]\times\R^{n-1})\right)
	\]
	under the map
	\[
	i_{a,*}\colon K_*\!\left(C^*([a-c,a+c]\times\R^{n-1})\right)\longrightarrow 
	K_*(C^*(\{a\}\times\R^{n-1}\subset X))
	\]
	induced by inclusion, for some $a\ll 0$ and  $c>0$. In particular, although $i_{a,*}(\ind_q(D,\psi_a))=0$ for every $a\in\R$, we nevertheless have $\ind_q(D,\psi_a)\neq 0$ for $a\ll 0$ by Theorem \ref{thm:quantitativeIndex}.
\end{example}

\subsection{An excision property}

Before proving Theorem \ref{thm:quantitativeIndex}, we first introduce an alternative representative of $\ind_q(D,\psi)$ which is equivalent to the original one but has the advantage of depending only on the geometry near $\mathcal N_\psi$. This leads to an excision property of the quantitative partitioned index, which will be used to prove Theorem \ref{thm:quantitativeIndex}.

Let $B$ be the operator \eqref{eq B odd} or \eqref{eq B even}, depending on $n$.  Given $b>0$, let    $F_b = \frac{bB}{\sqrt{1+(bB)^2}}$ as before.
\begin{lemma}\label{lem GbK}
Let $\delta \in (0,1/2)$.
Let  $b=r=\frac{1}{A\delta}$.
There is a bounded operator $G_{\delta}$ with propagation no more than $(A\delta^2)^{-1}$, such that $F_{(A\delta)^{-1}} - G_{\delta}$ is locally compact, and
\begin{align}
\|F_{(A\delta)^{-1}} - G_{\delta}\| &\leq \frac{2}{\pi} e^{-\delta^{-1}};  \label{eq Fb GbK}\\
    \|[G_{\delta},\sqrt{\chi}]\|
       & \le 12 \delta^{1/2} + 4 e^{-\delta^{-1}}; \label{eq G.1}\\
\label{eq G.2}
    \|(G_{\delta}^2 - 1)(1-\chi)\|
        &\le \delta+ 4 e^{-\delta^{-1}}.
\end{align}
Furthermore, $G_{\delta}^2 - 1$ lies in $M_2(C^*(X))$ if $n$ is odd and in $C^*(X)$ if $n$ is even.
\end{lemma}
\begin{proof}
The distributional Fourier transform of the function $x \mapsto x(1+x^2)^{-1/2}$ is
\[
 \xi \mapsto   \frac{2}{i}\,\sgn(\xi)\,K_1(|\xi|),
\]
where $\sgn$ denotes the sign function and $K_1$ is the modified Bessel function of the second kind. This function has a singularity at $\xi=0$, but is smooth and decays exponentially as $|\xi|\to\infty$.
Consequently,
\[
F_b 
    = \frac{1}{\pi i}\int_{\mathbb R} \sgn(\xi)\,K_1(|\xi|)\, e^{i\xi bB}\, d\xi
\]
in the weak operator topology (see Proposition 10.3.5 in \cite{Higson00} or Proposition D.2.3 in \cite{higherindex}).

For any $K\ge 2$, define the truncated operator
\[
    G_{b,K} \coloneqq 
    F_b - \frac{1}{\pi i}
        \int_{|\xi|\ge K} \sgn(\xi)\,K_1(|\xi|)\, e^{i\xi bB}\, d\xi
    = \frac{1}{\pi i}
        \int_{-K}^{K} \sgn(\xi)\,K_1(|\xi|)\, e^{i\xi bB}\, d\xi.
\]
This operator has propagation at most $bK$. 

Since $K_1$ decays exponentially and satisfies
\[
    |K_1(\xi)| \le e^{-|\xi|} \qquad (\xi\ge 2),
\]
the tail integral is absolutely convergent, and
\[
    \left\|
    \frac{1}{\pi i}
    \int_{|\xi|\ge K}
        \sgn(\xi)\,K_1(|\xi|)\,e^{i\xi bB}\, d\xi
    \right\|
    \le \frac{2}{\pi}e^{-K}.
\]
This implies that
 $\|F_b - G_{b,K}\| \leq \frac{2}{\pi} e^{-K}$. 
Analogously to the proof of Lemma~\ref{lem est Fb}, we find that
\[
\begin{split}
    \|[G_{b,K},\sqrt{\chi}]\|
       & \le 12\,\frac{\sqrt b}{r\sqrt A} + 4 e^{-K}; \\
    \|(G_{b,K}^2 - 1)(1-\chi)\|
        &\le \frac{1}{bA} + 4 e^{-K}.
\end{split}
\]
If  $\delta \in (0, 1/2)$, and we choose $r=b = (A\delta)^{-1}$ and $K = \delta^{-1}>2$, then the operator
$G_{\delta}:= G_{(A\delta)^{-1},\delta^{-1}}$ has properties \eqref{eq Fb GbK}--\eqref{eq G.2}.

The operator $G_{\delta}^2 - 1$ has finite propagation. Let $f \in C_0(X)$.
Using compactness of $(F_{(A\delta)^{-1}} - G_{\delta})f$ and $[F_{(A\delta)^{-1}}, f]$ (\cite[Lemma 10.6.4]{Higson00}), we see that
\[
(G_{\delta}^2-1)f\sim (F_{(A\delta)^{-1}}^2 -1)f \sim 0.
\]
where $\sim$ denotes equality modulo compact operators. Similarly, $f(G_{\delta}^2-1)$ is compact, so $G_{\delta}^2-1$ is locally compact. Hence the last statement follows.
\end{proof}

\begin{proposition}\label{prop excision}
Assume the same hypotheses as in Theorem~\ref{thm:quantitativeIndex}. Then there exists a universal constant $c_3>0$ and a representative of
\[
\ind_q(D,\psi)\in K_{n-1}\bigl(C^*(\bar B_{c_3A^{-1}}(\mathcal N_\psi))\bigr)
\cong
K_{n-1}\bigl(C^*(N,d_{c_3A^{-1}})\bigr)
\]
that depends only on the geometry of $\bar B_{c_3A^{-1}}(\mathcal N_\psi)$, namely if $\bar B_{c_3A^{-1}}(\mathcal N_\psi)$ isometrically embeds into two complete manifolds satisfying the assumptions in Theorem~\ref{thm:quantitativeIndex}, then their index classes as above coincide.
\end{proposition}

\begin{proof}
We first consider the case where $n$ is odd.
Let $\delta$, $b$, $r$ and $G_\delta$ be as in Lemma~\ref{lem GbK}. Write
\[
G_\delta=
\begin{pmatrix}
    0 & V_\delta \\
    U_\delta & 0
\end{pmatrix}.
\]
Let $\bar P$ be the operator defined from $G_\delta$ as in~\eqref{eq def tilde P}, namely,
\begin{equation}\label{eq:1000-barP}
\bar P
=
\begin{pmatrix}
    1 & 0 \\
    0 & 0
\end{pmatrix}
+
\sqrt{\chi}
\begin{pmatrix}
   -(1-U_\delta V_\delta)^2
   &
   (2-U_\delta V_\delta)U_\delta(1-V_\delta U_\delta)
   \\
   V_\delta(1-U_\delta V_\delta)
   &
   (1-V_\delta U_\delta)^2
\end{pmatrix}
\sqrt{\chi}.
\end{equation}
By the last point in Lemma \ref{lem GbK}, this operator lies in $M_2(C^*(X)^+)$.  
Because $\supp(\chi) \subset \bar B_{(A\delta)^{-1}}(\mathcal{N}_{\psi})$, in fact
\[
\bar P-
\begin{pmatrix}
    1 & 0 \\
    0 & 0
\end{pmatrix}
\in
M_2\bigl(C^*(\bar B_{(\delta A)^{-1}}(\mathcal N_\psi))\bigr).
\]

Arguing as in the proof of Proposition \ref{prop quasi proj}, using \eqref{eq G.1} instead of \eqref{eq1.1} and \eqref{eq G.2} instead of \eqref{eq1.2}, we find that for $\delta$ small enough,
\[
\|\bar P\|\leq 3 \quad \text{and} \quad
\|\bar P^2-\bar P\|<\frac{1}{1000}.
\]
Moreover, the propagation of $\bar P$ is bounded by $5(A\delta^2)^{-1}$, since the entries of \eqref{eq:1000-barP} are compositions of local operators, and at most $5$ matrix components of $G_{\delta}$.
Set
\[
c_3^0=\delta^{-1}+5\delta^{-2}.
\]
We may then regard $\bar P$ as an element of
\[
M_2\bigl(C^*(\bar B_{c_3^0A^{-1}}(\mathcal N_\psi))^+\bigr),
\]
and in particular, $\bar P$ depends only on the geometry of $\bar B_{c_3^0A^{-1}}(\mathcal N_\psi)$.

Using \eqref{eq Fb GbK} and Lemma \ref{lemma:Theta-P}, we can ensure that  $\|\Theta(\bar P) - \Theta(\tilde P)\|$ is smaller than (say) $\frac{1}{100}$, by taking $\delta$ small enough. Then
 the idempotents $\Theta(\bar P)$ and $\Theta(\tilde P)$ are homotopic. Therefore the index class $\ind_q(D,\psi)$ is also represented by
\[
[\Theta(\bar P)]-
\left[
\begin{pmatrix}
    1 & 0 \\
    0 & 0
\end{pmatrix}
\right]
\in
K_0\bigl(C^*(\bar B_{c_3^0A^{-1}}(\mathcal N_\psi))\bigr),
\]
where the representative only depends on the geometry of $\bar B_{c_3^0A^{-1}}(\mathcal N_\psi)$.

\medskip
Now suppose that $n$ is even. Recall that the index class is represented by the operator $\tilde Q$ in \eqref{eq def tilde Q}.
With $a_k$ as in Lemma \ref{lem exp series}, we write for $N \in \N$, 
\begin{equation}
\tilde Q_N :=
1
+
\sqrt{\chi}
\left(\frac 1 4(F_b^2-1)\sum_{k=0}^N a_k \left(\frac{F_b+1}{2}\right)^k
\right)
\sqrt{\chi}.
\end{equation}
Then by \eqref{eq Qb - 1},  $\|\tilde Q - \tilde Q_N\|$ decays faster than any exponential function of $N$, uniformly in $b$ since $\|F_b\|\leq 1$. Choose $N$ such that $\|\tilde Q - \tilde Q_N\|<1/2000$.

We define
\begin{equation}
\bar Q
=
1
+
\sqrt{\chi}
\left(\frac 1 4(G_\delta^2-1)\sum_{k=0}^N a_k \left(\frac{G_\delta+1}{2}\right)^k
\right)
\sqrt{\chi}.
\end{equation}
This operator lies in $C^*(X)^+$ by the last point in Lemma \ref{lem GbK}. 
By \eqref{eq Fb GbK}, we can choose $\delta$ small enough such that $\|\bar Q - \tilde Q_N\|<1/2000$. 
Then
\beq{eq bar Q tilde Q}
\|\bar Q-\tilde Q\|<\frac{1}{1000},
\eeq
and hence $\bar Q$ is invertible. 
Moreover, the propagation of $\bar Q$ is bounded by $N(\delta^2A)^{-1}$, by Lemma \ref{lem GbK}.
Set
\[
c_3^1=\delta^{-1}+N\delta^{-2}.
\]
Then $\bar Q$ may be regarded as an element of
$
C^*(\bar B_{c_3^1A^{-1}}(\mathcal N_\psi))^+$,
and depends only on the geometry of $\bar B_{c_3^1A^{-1}}(\mathcal N_\psi)$.
In this case, \eqref{eq bar Q tilde Q} implies that the index class $\ind_q(D,\psi)$ is represented by
\[
[\bar Q]
\in
K_1\bigl(C^*(\bar B_{c_3^1A^{-1}}(\mathcal N_\psi))\bigr),
\]
which again depends only on the geometry of $\bar B_{c_3^1A^{-1}}(\mathcal N_\psi)$.

Finally, the proof is completed if we set
$c_3=\max\{c_3^0,c_3^1\}.$
\end{proof}

\begin{remark}
The proof of the partitioned manifold index theorem for noncompact hypersurfaces in \cite[Subsection 4.5]{HochsdeKok25} is based on the idea that this index is ``determined near the partitioning hypersurface $N$''. This is used to reduce the problem to the case where the ambient manifold is $N \times \R$ (\cite[Theorem 3.1]{HochsdeKok25}, or Proposition \ref{prop:product}). This argument is based on the assumption \eqref{eq cond N equiv} that the metric spaces $(N, d_{g_N})$ and $(N, d_g|_N)$ are coarsely equivalent (used in the proof of \cite[Proposition 4.16]{HochsdeKok25}). Proposition \ref{prop excision} is a quantitative refinement of this idea, where this assumption is not needed. 
\end{remark}
\subsection{Proof of the quantitative partitioned manifold index theorem}

\begin{proof}[Proof of Theorem \ref{thm:quantitativeIndex}]
   Let $B_\varepsilon(N)=U\cong N \times (-\varepsilon,\varepsilon)$ be the uniform geodesic normal neighborhood of $N$ in $X$ as in Theorem \ref{thm:uniformTubular}. Choose $\varepsilon$ small enough so that $B_{\varepsilon}(N) \subset \psi^{-1}((-A,A))$. Let $\varphi\colon X \to [-1,1]$ be a smooth function such that 
   \begin{itemize}
   \item $|\varphi(x)|=1$ for all $x \in X \setminus B_{\varepsilon/4}(N)$;
   \item if $\psi(x) = \pm A$, then $\varphi(x) = \pm 1$, with the same signs;
   \item on $U \cong  N \times  (-\varepsilon,\varepsilon)$, the function $\varphi$ is an function of the parameter in $(-\varepsilon, \varepsilon)$, and constant in the factor $N$.
   \end{itemize}
    
For $u,v\geq 0$, define
\[
\psi_{u,v}=u\varphi+v\psi.
\]
Note that
\[
\psi_{u,v}=\pm(u+Av)
\quad\text{on}\quad
\{p\in X:\ |\psi(p)|=A\}.
\]
So
\[
\psi_{u,v}^{-1}( (-(u+vA, u+vA))) \subset \psi^{-1}((-A,A)) \subset \mathcal{N}_{\psi}.
\]
Therefore, as in Definitions~\ref{def quant index odd} and~\ref{def quant index even}, with $b$ and $r$ as in Remark \ref{rem O A-1}, the pair $(D,\psi_{u,v})$ defines an index class
\[
\ind_q(D,\psi_{u,v})
\in
K_{n-1}\bigl(C^*(\bar B_{(\delta(u+Av))^{-1}}(\mathcal N_\psi))\bigr)
\]
whenever $u$ and $v$ are not simultaneously zero.
By homotopy invariance and part (c) of Lemma \ref{lem ind indep choices}, the class $\ind_q(D,\psi_{u,v})$ is independent of $u$ and $v$ when mapped into
\[
K_{n-1}\bigl(C^*(\bar B_{(\delta A)^{-1}}(\mathcal N_\psi))\bigr),
\]
provided that $u+Av\geq A$.

Now consider the case $v=0$, so that $\psi_{u,0}=u\varphi$.
Define
\[
\mathcal N_\varphi\coloneqq \{p\in X:\ |\varphi(p)|<1\} \subset B_{\varepsilon/4}(N)
\cong N \times
(-\varepsilon/4,\varepsilon/4).
\]
Let $c_3$ be as in Proposition~\ref{prop excision}.
If $u$ is sufficiently large so that $c_3u^{-1}<\varepsilon/4$ and $u>A$, then we have
\[
\ind_q(D,u\varphi)
\in
K_{n-1}\bigl(C^*(\bar B_{c_3 u^{-1}}(\mathcal{N}_{\varphi}))\bigr),
\]
and by Proposition~\ref{prop excision}, this class depends only on the geometry of 
\[
\bar B_{c_3 u^{-1}}(\mathcal{N}_{\varphi}) \subset 
\bar B_{\varepsilon/2}(N).
\]
It follows from homotopy invariance and part (c) of Lemma \ref{lem ind indep choices} that
\begin{equation}\label{eq:indq_D_u_phi}
\ind_q(D,\psi)
=
(i_{ \bar B_{c_3 u^{-1}}(\mathcal{N}_{\varphi}) } )_*\bigl(\ind_q(D,u\varphi)\bigr) \quad \in K_{n-1}(C^*(\bar B_{(\delta A)^{-1}}(\mathcal N_\psi))),
\end{equation}
where
\[
i_{\bar B_{c_3 u^{-1}}(\mathcal{N}_{\varphi})}\colon \bar B_{c_3 u^{-1}}(\mathcal{N}_{\varphi})\longrightarrow \bar B_{(\delta A)^{-1}}(\mathcal N_\psi)
\]
is the inclusion map.

By part 2 of Theorem \ref{thm:uniformTubular}, the distance function on $U$ induced by $g$ is bi-Lipschitz equivalent to the distance function induced by the product metric.
Therefore, the inclusion-induced map
\[
K_*(C^*(N))\rightarrow K_*(C^*(\bar B_{c_3 u^{-1}}(\mathcal{N}_{\varphi})))
\]
is an isomorphism, and hence we may regard
\[
\ind_q(D,u\varphi)\in K_*\bigl(C^*(N)\bigr).
\]

By part 2 of Theorem \ref{thm:uniformTubular} and \cite[Theorem~1.9]{MR4937352}, there exists a complete Riemannian metric $\bar g$ with bounded geometry on $N \times \R$ satisfying:
\begin{itemize}
    \item $\bar g=g$ on $N \times (-\varepsilon/2,\varepsilon/2)$;
    \item $\bar g=dt^2+g_N$ for $|t|>\varepsilon$;
    \item $\bar g$ is bi-Lipschitz equivalent to $dt^2+g_N$.
\end{itemize}
Using the same construction as above and in Proposition~\ref{prop excision}, we obtain an index class
\[
\ind_q(\bar D,u\varphi)
\in
K_*\bigl(C^*(N \times [-\varepsilon/2,\varepsilon/2],d_{\bar g})\bigr)
\cong
K_*\bigl(C^*(N)\bigr),
\]
where $\bar D$ denotes the Dirac operator associated with $\bar g$.
By the excision property in Proposition~\ref{prop excision}, we have
\beq{eq excision u phi}
\ind_q(\bar D,u\varphi)=\ind_q(D,u\varphi)\in K_*\bigl(C^*(N)\bigr).
\eeq

Again by \cite[Theorem~1.9]{MR4937352}, there exists a path of complete Riemannian metrics $\{g_s\}_{s\in[0,1]}$ with uniformly bounded geometry on $N \times \R$ such that
\[
g_0=dt^2+g_N,
\qquad
g_1=\bar g,
\]
and all $g_s$ are uniformly bi-Lipschitz equivalent to $dt^2+g_N$.
By homotopy invariance and \eqref{eq excision u phi}, we obtain
\beq{eq ind q ind N}
\ind_q(D_{N\times\R},u\varphi)
=
\ind_q(\bar D,u\varphi)=
\ind_q(D,u\varphi)
\in
K_*\bigl(C^*(N)\bigr),
\eeq
where $D_{N\times\R}$ denotes the Dirac operator associated with the product metric. 
By  
 Lemma \ref{lem part 1}, Proposition~\ref{prop:product},   and the form of the isomorphism \eqref{eq iso loc Roe} used there,  the left hand side of \eqref{eq ind q ind N} equals $\ind(D_N)$. By \eqref{eq:indq_D_u_phi}, we conclude that
\[
\ind_q(D,\psi)
=
(i_{\bar B_{c_3 u^{-1}}(\mathcal{N}_{\varphi})})_*\bigl(\ind(D_N)) \quad \in K_{n-1}(C^*(\bar B_{(\delta A)^{-1}} (\mathcal{N}_{\psi})  )).
\]
By functoriality in Theorem \ref{thm:coarseMap}, the right hand side equals $(\id)_*(\ind(D_N))$. So Theorem \ref{thm:quantitativeIndex} follows, with $c_2 = \delta^{-1}$.
\end{proof}

\section{Proof of the noncompact band-width inequality} \label{sec proof bandwidth}

In this section, we prove Theorem \ref{thm:bandwidthForBand}, and illustrate with an example that condition 1 in this theorem is necessary.

\subsection{Proof of Theorem \ref{thm:bandwidthForBand}} 

We begin with a band-width inequality for complete manifolds.

\begin{theorem}\label{thm noncpt bandwidth}
	Let $(X^n,g)$ be a complete $\Spin$ Riemannian manifold with bounded geometry, partitioned by a uniformly embedded hypersurface $N^{n-1}$. 
	Given $\ell>0$ and $a\in(-\ell/2,\ell/2)$, consider the $1$-Lipschitz signed distance function
	\[
	\rho(x)=
	\begin{cases}
		\min\{\,a + d_g(x,N),\, \ell/2\,\}, & x\in X_+ \cup N,\\[0.3em]
		\max\{\,a - d_g(x,N),\, -\ell/2\,\}, & x\in X_-.
	\end{cases}
	\]
	Suppose that
	\begin{enumerate}
		\item for every $0<t<\ell/2$, the identity map 
		\[
		\id \colon (N,d_{g_N}) \longrightarrow (N,d_{\rho^{-1}((-t,t))})
		\]
		is a coarse equivalence, where $d_{\rho^{-1}((-t,t))}$ is the restriction of the Riemannian distance on $\rho^{-1}((-t,t))$, and
		\item the Dirac operator $D_N$ has non-zero index in $K_{n-1}(C^*(N))$.
	\end{enumerate}
	Then
	\beq{eq band width X}
	\inf_{x\in X; |\rho(x)|<\ell/2}\Sc_g(x)\leq \frac{4\pi^2(n-1)}{n\ell^2}.
	\eeq
\end{theorem}

Let $\kappa=\inf_{x\in X; |\rho(x)|<\ell/2}\Sc_g(x)$. 
We assume that
	\beq{eq assumpt scalar}
	\frac{\kappa n}{4(n-1)}>\frac{\pi^2}{\ell^2}
	\eeq	
and deduce a contradiction, so that 	\eqref{eq band width X} follows.

	Choose $\ell_1,\ell_2>0$ and $\delta>0$ such that
	\[
	\ell>\ell_1>\ell_2, \qquad 
	\frac{\kappa n}{4(n-1)}>\frac{\pi^2}{\ell_2^2}+\delta,
	\qquad |a|<\ell_2/2.
	\]
	By assumption, $\rho^{-1}((-\ell_1/2,\ell_1/2))$ lies in the $\ell_1$–neighborhood of $N$.
	\begin{lemma} \label{lem psi exist} 
For sufficiently small $\varepsilon>0$, there exist $A_\varepsilon \geq \frac{1}{2\varepsilon}$ and a smooth function $\psi$ on $X$ such that
$
\psi\equiv \pm A_\varepsilon
$
 on $X\setminus \rho^{-1}((-\ell_2/2,\ell_2/2))$, and
		\[
		\frac{n}{4(n-1)}\Sc_g + \psi^2 - |\nabla\psi|
		\ge \frac{\delta}{2}
		\]
on $X$.
\end{lemma}
\begin{proof}
Set
	\[
	r=\frac{\pi(1+\varepsilon)}{\ell_2-2\varepsilon}.
	\]
	Choose a smooth non-negative function 
	\(\xi:\mathbb R^2\to\mathbb R\)
	satisfying the following assumptions for all $(x,y) \in \R^2$:
	\begin{align}
		 0&\le (1+\varepsilon)\xi(x,y)\le y^2+r^2;\label{eq ass xi 1}\\
		 \xi(-x,y)&=\xi(x,y);\nonumber\\
		\xi(x,y)&=
		\begin{cases}
			\dfrac{y^2+r^2}{1+\varepsilon}, & |x|\le\ell_2/2-2\varepsilon,\\[0.4em]
			0, & |x|\ge \ell_2/2-\varepsilon.
		\end{cases}\nonumber
	\end{align}

	Let $f$ be the unique solution to the following differential equation with initial condition:
	\begin{equation}\label{eq:ode}
		\begin{cases}
			f'(x)=\xi(x,f(x)),\\
			f(0)=0.
		\end{cases}
	\end{equation}
	The solution always exists in a small neighborhood of $x=0$. 
	By the comparison theorem and the assumption \eqref{eq ass xi 1}, we have for all $x \in \bigl(-(l_2/2 - \varepsilon), l_2/2 - \varepsilon \bigr)$,
	\begin{equation}
		f(x)\leq r\tan\left(\frac{r x}{1+\varepsilon}\right),
	\end{equation}
	since the function $h(x)= r\tan(\frac{r x}{1+\varepsilon})$ is the unique solution to the differential equation 
	$$\begin{cases}
		h'(x)=\frac{h(x)^2+r^2}{1+\varepsilon},\\
		h(0)=0.
	\end{cases}$$
	Therefore, the solution $f$ exists at least when $|x|<\ell_2/2 - \varepsilon$. In particular, if $|x|\leq \ell_2/2-2\varepsilon$, then $f(x)=r\tan(\frac{rx}{1+\varepsilon})$, and 
	\beq{eq value f}
	f(\ell_2/2-2\varepsilon)=\frac{\pi(1+\varepsilon)}{\ell_2-2\varepsilon}\tan\left(\frac{\pi}{2}\cdot\frac{\ell_2/2-2\varepsilon}{\ell_2/2-\varepsilon}\right).
	\eeq
    Since by construction $f'(x)=0$ if $|x|\geq \ell_2/2-\varepsilon$, we find that $f$ is defined for all $x\in\R$. In particular, $f(x)=\pm A_\varepsilon$ when $|x|\geq \ell_2/2-\varepsilon$, where, by \eqref{eq value f} and nonnegativity of $f'$,
   \beq{eq A eps}A_\varepsilon\geq \frac{\pi(1+\varepsilon)}{\ell_2-2\varepsilon}\tan\left(\frac{\pi}{2}\cdot\frac{\ell_2/2-2\varepsilon}{\ell_2/2-\varepsilon}\right).
   \eeq
   Using the fact that for all $a>0$, we have 
\[	
	\tan\left(\frac{\pi}{2}\frac{a-2\varepsilon}{a-\varepsilon} \right) = \frac{2a}{\pi \varepsilon} + O(1) 
\]	
as $\varepsilon\downarrow 0$, we find that the right hand side of \eqref{eq A eps} is $\frac{1}{\varepsilon} + O(1)$.
    So $A_{\varepsilon}> \frac{1}{2\varepsilon}$, if $\varepsilon>0$ is small enough.

Let $\rho_{\varepsilon}$ be a smooth function on $X$ such that
\begin{align}
\|\rho - \rho_{\varepsilon}\|_{\infty} &< {\varepsilon};\label{eq rho eps 1}\\
\|\nabla \rho_{\varepsilon}\|_{\infty}&<(1+ \varepsilon).\label{eq rho eps 2} 
\end{align}
Define $\psi(p)=f(\rho_{\varepsilon}(p))$ for all $p\in X$.
	\begin{itemize}
		\item If $p\in\rho^{-1}((-\ell_2/2,\ell_2/2))$, then $\Sc_{g}(p)\geq\kappa>0$. Therefore, by \eqref{eq rho eps 2}, 
        \begin{align*}
            \left(\frac{ n}{4(n-1)}\Sc_{ g}+\psi^2-|\nabla\psi| \right)(p)
            &\geq\frac{\kappa n}{4(n-1)}+f(\rho_{\varepsilon}(p))^2-(1+\varepsilon)f'(\rho_{\varepsilon}(p))\\
            &=\frac{\kappa n}{4(n-1)}+f(\rho_{\varepsilon}(p))^2-(1+\varepsilon)\xi(\rho_{\varepsilon}(p),f(\rho_{\varepsilon}(p)))\\
           & \geq 
           \frac{\kappa n}{4(n-1)}-{r^2}\\
            &=\frac{\kappa n}{4(n-1)}-\frac{\pi^2(1+\varepsilon)^2}{(\ell_2-2\varepsilon)^2},
        \end{align*}
        where we used \eqref{eq ass xi 1} in the third line.
The right hand side is greater than or equal to $\delta/2$ when $\varepsilon$ is small enough.
		\item If $p\notin\rho^{-1}((-\ell_2/2,\ell_2/2))$, then 
		 $|\rho_{\varepsilon}(p)|>\ell_2/2-\varepsilon$ by \eqref{eq rho eps 1}. So $f(\rho_{\varepsilon}(p))=\pm A_\varepsilon$ and $\nabla\psi(p)=0$. 
	As $X$ has bounded geometry, we may assume $\Sc_g\ge -C$ on $X$, for some $C>0$.		 
		 Then
		$$\left(\frac{ n}{4(n-1)}\Sc_{ g}+\psi^2-|\nabla\psi|\right)(p)\geq 
        \frac{ n}{4(n-1)}(-C)+A_\varepsilon^2,$$
        which is greater than or equal to $\delta/2$ when $\varepsilon$ is small enough.
	\end{itemize}
\end{proof}

Let $D$ be the Dirac operator on $X$, and let $\varepsilon$ and $\psi$ be as in Lemma \ref{lem psi exist}. 
\begin{lemma}\label{lem D psi invtble}
We have
\beq{eq B invtble}
(D\pm i \psi)^*(D \pm i \psi) \geq \frac{\delta}{2}.
\eeq
\end{lemma}
\begin{proof}
 For any compactly supported smooth spinor $\sigma$ on $X$, we have
\begin{align*}
\int_X|(D\pm i\psi)\sigma|^2 d\vol_g &=\int_X \bigl( |D\sigma|^2+\psi^2|\sigma|^2\pm\langle i(D\psi-\psi D)\sigma,\sigma\rangle \bigr)d\vol_g\\
&=\int_X \bigl(  |D\sigma|^2+\psi^2|\sigma|^2\pm \langle ic(\nabla\psi)\sigma,\sigma\rangle \bigr) d\vol_g\\
&\geq \bigl( \int_X |D\sigma|^2+\psi^2|\sigma|^2-|\nabla\psi|^2|\sigma|^2 \bigr) d\vol_g.
\end{align*}
By the Lichnerowicz formula, we have
$$\int_X|D\sigma|^2 d\vol_g =\int_X \bigl( |\nabla\sigma|^2+\frac{\Sc_g}{4}|\sigma|^2\bigr) d\vol_g,$$
and by the Cauchy--Schwarz inequality, we have
$$|D\sigma|^2=|\sum_{j=1}^n c(e_j)\nabla_{e_j}\sigma|^2\leq n|\nabla\sigma|^2.$$
Therefore,
$$\int_X|D\sigma|^2 d\vol_g \geq \frac 1 n\int_X|D\sigma|^2 d\vol_g+\int_X \frac{\Sc_g}{4}|\sigma|^2 d\vol_g.$$
Hence
$$\int_X|D\sigma|^2 d\vol_g \geq  \frac{n}{n-1}\int_X\frac{\Sc_g}{4}|\sigma|^2 d\vol_g.$$
By these computations and Lemma \ref{lem psi exist}, 
\[
\begin{split}
\int_X|(D\pm i\psi)\sigma|^2 d \vol_g &\geq\int_X \left(\frac{n}{n-1}\cdot\frac{\Sc_g}{4}
+\psi^2-|\nabla\psi|
\right)|\sigma|^2 d \vol_g\\
&\geq\frac{\delta}{2}\int_X|\sigma|^2 d \vol_g.
\end{split}
\]
Hence \eqref{eq B invtble} follows.
\end{proof}

\begin{proof}[Proof of Theorem \ref{thm noncpt bandwidth}]
Let 
\beq{eq choice N psi}
\mathcal N_\psi\coloneqq \overline{ \rho^{-1}((-\ell_2/2,\ell_2/2))}, 
\eeq
which is independent of $\varepsilon$.
Applying Definition \ref{def quant index odd} or \ref{def quant index even},  with $\delta$ as in Proposition \ref{prop quasi proj} or \ref{prop tilde Q Qb}, and $b$ and $r$ as in Remark \ref{rem O A-1}, we obtain
\[
\ind_q(D,\psi)\in K_{n-1}(C^*(\bar B_{(\delta A_\varepsilon)^{-1}}(\mathcal N_\psi))).
\]
 Theorem~\ref{thm:quantitativeIndex} shows that 
 \beq{eq qPMIT band}
 \ind_q(D,\psi) = \id_*(\ind(D_N)) \in K_{n-1}(C^*(\bar B_{(\delta A_\varepsilon)^{-1}}(\mathcal N_\psi))).  
 \eeq

%
%


By Lemmas \ref{lem part 2} and  \ref{lem D psi invtble}, we have for $A_{\varepsilon} \geq a \sqrt{\delta/2}$, 
\beq{eq ind q D 2}
 \ind_q(D,\psi)=0\in K_{n-1}(C^*(\bar B_{2^{1/4}c_1\delta^{-1/4} A_\varepsilon^{-1/2}}(\mathcal N_\psi))).
 \eeq
As $\delta$ is  independent of $\varepsilon$, we have
\[
2^{1/4}c_1\delta^{-1/4} A_\varepsilon^{-1/2} \geq (\delta A_\varepsilon)^{-1}
\]
if $A_{\varepsilon}$ is large enough, i.e.\ when $\varepsilon$ is small enough. Choose $\varepsilon$ in this way, and consider the inclusion map
\[
i_{\bar B_{(\delta A_\varepsilon)^{-1}}(\mathcal N_\psi)}
\colon \bar B_{(\delta A_\varepsilon)^{-1}}(\mathcal N_\psi) \to \bar B_{2^{1/4}c_1\delta^{-1/4} A_\varepsilon^{-1/2}}(\mathcal N_\psi).
\]
By part (c) of Lemma \ref{lem ind indep choices}, \eqref{eq qPMIT band} and \eqref{eq ind q D 2},
\begin{multline}\label{eq index vanishes}
(i_{\bar B_{(\delta A_\varepsilon)^{-1}}(\mathcal N_\psi)})_* \circ \id_*(\ind(D_N))=  \ind_q(D,\psi)=0\\
\in K_{n-1}(C^*(\bar B_{2^{1/4}c_1\delta^{-1/4} A_\varepsilon^{-1/2}}(\mathcal N_\psi))).
\end{multline}
 Now 
 \beq{eq incl compos}
 i_{\bar B_{(\delta A_\varepsilon)^{-1}}(\mathcal N_\psi)} \circ \id
 \eeq
 is the inclusion of $N$ into $\bar B_{2^{1/4}c_1\delta^{-1/4} A_\varepsilon^{-1/2}}(\mathcal N_\psi)$. For $\varepsilon$ small enough, the latter set is contained in $\rho^{-1}((-\ell_1/2,\ell_1/2))$. Hence \eqref{eq incl compos} is a coarse equivalence by the first assumption in Theorem \ref{thm noncpt bandwidth}. So the induced map on $K$-theory is an isoorphism by Theorem \ref{thm:coarseMap}. Therfeore, \eqref{eq index vanishes} implies that $\ind(D_N) = 0$, contradicting the second assumption in Theorem \ref{thm noncpt bandwidth}.
We conclude that  \eqref{eq assumpt scalar} is not true.
\end{proof}

\begin{proof}[Proof of Theorem \ref{thm:bandwidthForBand}]
    By Theorem \ref{thm:uniformTubular}, there is a uniform geodesic normal neighborhood $U\cong [0,\varepsilon)\times \partial_- M$ of $\partial_- M$ in $M$. Let $\varepsilon_1>0$ be smaller than $\varepsilon$ and $\ell/2$, and set $N=\{\varepsilon_1\}\times\partial_- M$. 
    Let $(X,g)$ be the complete manifold as in Definition \ref{def:band}. Then $X$ decomposes as
    $$X=X_1\cup_{\partial_- M} M\cup_{\partial_+ M} X_2,$$
    which is partitioned by $N$ into $X_\pm$.

    Let the function $\rho$ be as in Theorem \ref{thm noncpt bandwidth}, for $a = \frac{\ell}{2}+ \varepsilon_1$, and with $\ell$ in Theorem \ref{thm noncpt bandwidth} replaced by $\ell-\varepsilon_1$, for $\ell$ as in the current situation.
     The first condition in Theorem \ref{thm noncpt bandwidth} holds by the first condition in Theorem \ref{thm:bandwidthForBand}. 
      And the second condition  in Theorem \ref{thm noncpt bandwidth} holds by the second condition in Theorem \ref{thm:bandwidthForBand} and homotopy invariance of the coarse index. (This homotopy invariance is well-known, but one can also apply the more recent cobordism invariance result in  \cite{Wulff12}, or the later versions in \cite[Theorem 4.12]{Wulff19} or \cite[Corollary 2.27]{HochsdeKok25}.) Theorem \ref{thm noncpt bandwidth} therefore implies that
 \beq{eq rho eps}
 \inf_{|\rho(x)|<(\ell-\varepsilon_1)/2} \Sc_g(x) \leq \frac{4\pi^2(n-1)}{n (\ell-\varepsilon_1)^2}.
 \eeq   
 We claim that 
 \beq{eq rho M}
 \rho^{-1}\left( -\left(\frac{\ell - \varepsilon_1}{2} \right), \frac{\ell - \varepsilon_1}{2}  \right) \subset M.
    \eeq
 As    \eqref{eq rho eps} holds for arbitrarily small $\varepsilon_1$, we then obtain
 \eqref{eq band width ineq}.  
  
  Let $x \in X$, with $|\rho(x)|<\frac{\ell - \varepsilon_1}{2}$. Suppose first that $x \in X_+$. Then
  \[
  \rho(x) = \frac{\ell}{2}+ \varepsilon_1+ d_g(x, N),
  \]
  so
  \[
  d_g(x, N) < \ell - \frac{3\varepsilon_1}{2},
  \]
  which implies that $x \in M$. Similarly, if  $x \in X_-$, then we have
  \[
  \rho(x) = \frac{\ell}{2}- \varepsilon_1+ d_g(x, N),
  \]
  so
  \[
  d_g(x, N) < \frac{\varepsilon_1}{2},
  \]
and again $x \in M$. So \eqref{eq rho M}, and hence Theorem \ref{thm:bandwidthForBand}, follows.
   \end{proof}

\begin{remark}\label{rem motivation quant index}
In the setting of Theorem \ref{thm:bandwidthForBand}, with $X$ as in Definition \ref{def:band}, condition \eqref{eq cond N equiv} in Theorem \ref{thm:recallPartitionIndex} may not hold. For example, in Example \ref{ex quant index thm}, the hypersurface $\rho^{-1}(-1)$ does not have this property. In such cases, Theorem \ref{thm:recallPartitionIndex} does not apply. But condition 1 in Theorem \ref{thm:bandwidthForBand} does imply condition 1 in Theorem \ref{thm noncpt bandwidth}, for small enough $t$, where $N$ is as in the proof of  Theorem \ref{thm:bandwidthForBand}. This allows us to use Theorem \ref{thm noncpt bandwidth}, which is based on Theorem \ref{thm:quantitativeIndex},  to prove Theorem \ref{thm:bandwidthForBand}. This is our motivation for generalising Theorem \ref{thm:recallPartitionIndex} to Theorem \ref{thm:quantitativeIndex}. 
\end{remark}

\subsection{Necessity of condition 1}\label{sec cond 1 necessary}

We end this paper by discussing an example showing that condition 1 in Theorem~\ref{thm noncpt bandwidth} and Theorem~\ref{thm:bandwidthForBand} is indeed necessary. We start by recalling a proposition about positive scalar curvature metrics on connected sums, which is essentially due to Gromov--Lawson \cite[Proposition 5.1]{GromovLawson80} and Schoen--Yau \cite[Corollary 7]{SchoenYau79}. The following quantitative version is  \cite[Proposition 4.2]{Sweeney}.
\begin{proposition}\label{prop:connectedSum}
    Let $(M^n,g)$ be a Riemannian manifold with $n\geq 3$ and $p,p'\in M$. Assume that $\Sc_g\geq\kappa$. Then for any $d,\delta,\varepsilon>0$ there exists a Riemannian manifold $(P^n,g')$, which one constructs by removing two $\delta$-balls near $p$ and $p'$ and attaching a cylinder $T\cong S^{n-1}\times[0,1]$, namely
    $$P=(M\backslash (B_\delta(p)\cup B_\delta(p',))\cup T,$$
    such that
    \begin{enumerate}
        \item $\Sc_{g'}\geq \kappa-\varepsilon,$
        \item $g$ and $g'$ coincide on the $[\delta/2,\delta]$-annulus aroud $p$ and $p'$, and
        \item the diameter of $T$ is $\leq C(d+\delta)$, where $C$ is a universal constant.
    \end{enumerate}
\end{proposition}

Now
	fix $n\ge 3$ and $\delta>0$ a small positive number.	
	Consider the warped product metric
	\[
	g_\varphi = dt^2 + \varphi(t)^2 g_{\mathrm{eu}}
	\]
	on $[-\delta,\delta]\times\mathbb R^{\,n-1}$. Its scalar curvature is
	\[
	\Sc_{g_\varphi}
	= -\frac{(n-1)(n-2)\varphi'^2}{\varphi^2} - \frac{2(n-1)\varphi''}{\varphi}.
	\]

    Take 
    $$\varphi(t)=\left(\cos\bigl(\frac{nt}{2}\bigr)\right)^{2/n},$$ 
    and we obtain
	\[
	\Sc_{g_\varphi}=n(n-1)>0.
	\]
	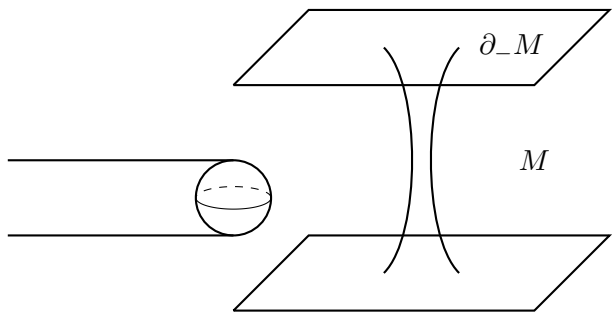
\begin{figure}[h]
    \centering
    \begin{tikzpicture}
        \draw[black,thick] (0,0) -- (4,0) -- (5,1) -- (1,1) -- (0,0);
        \draw[black,thick] (0,3) -- (4,3) -- (5,4) -- (1,4) -- (0,3);
        \draw[black,thick] (2,0.5) .. controls (2.5,1) and (2.5,3).. (2,3.5);
        \draw[black,thick] (3,0.5) .. controls (2.5,1) and (2.5,3).. (3,3.5);
        \draw[black,thick] (0,1) arc(-90:270:0.5);
        \draw[smooth, samples=100, domain=0:180,dashed] plot({0.5*cos(\x)},{1.5+0.15*sin(\x)});
        \draw[smooth, samples=100, domain=180:360] plot({0.5*cos(\x)},{1.5+0.15*sin(\x)});
        \draw[black,thick] (0,1)-- (-3,1);
         \draw[black,thick] (0,2)-- (-3,2);
         \draw (4,2) node {$M$};
         \draw (3.7,3.5) node {$\partial_- M$};
    \end{tikzpicture}
    \caption{Band with infinite length, positive scalar curvature, and non-zero index}
    \label{fig}
\end{figure}

	Let $O\in\mathbb R^{n-1}$ be the origin. Choose a small tube (homeomorphic to $[-\delta,\delta]\times S^{n-2}$) near $[-\delta,\delta]\times\{O\}$ which smoothly joins the slices $\{\pm\delta\}\times\mathbb R^{n-1}$. 
	Let the open subset $M_1 \subset [-\delta, \delta] \times \R^{n-1}$ be the exterior of this tube. We view it as a Riemannian manifold, with the restriction of $g_{\varphi}$. 
Its boundary is
	\[
	N\coloneqq \partial M_1
	= (\{\delta\}\times\mathbb R^{n-1})\#(\{-\delta\}\times\mathbb R^{n-1})
	\]
	is the connected sum of two copies of $\mathbb R^{n-1}$. (See Figure~\ref{fig}.)
	
	Pick a point $p=(0,p_1)\in M_1$ away from the tube and, for sufficiently small $r>0$, denote by $B_r(p)$ the geodesic $r$-ball about $p$ with boundary $S_r(p)$. Let $Y$ be the manifold by removing $B_r(p)$ from $M_1$ and attaching a cylinder 
		\[
	Y \coloneqq (S^{n-1}\times(-\infty,0])\cup_{S_r(p)}(M_1\setminus B_r(p)),
	\]
	By Proposition~\ref{prop:connectedSum}, there is a complete Riemannian metric $g$ on $Y$ with bounded geometry, which equals $g_\varphi$ away from $p$ inside $M_1$, equals the product metric $dt^2+g_{S^{n-1}}$ on the end $S^{n-1}\times(-\infty,0]$ away from the glued region, and satisfies 
	$$\Sc_g\ge n(n-1)-1>0.$$
	
	For any $R\gg 0$, define the band
	\[
	M = (S^{n-1}\times[-R,0])\cup_{S_\delta(p)}(M_1\setminus B_\delta(p))
	\]
	equipped with the induced metric $g$. Then $M$ is a regular $\Spin$ Riemannian band with
	\[
	\partial_-M = \mathbb R^{n-1}\#\mathbb R^{n-1}, \qquad \partial_+M = \{-R\} \times S^{n-1},
	\]
	and $\Sc_g\ge n(n-1)-1>0$ on $M$. The coarse index of the Dirac operator on $\partial_-M=\mathbb R^{n-1}\#\mathbb R^{n-1}$ is non-zero: under the identification
	\[
	K_{n-1}\bigl(C^*(\mathbb R^{n-1}\#\mathbb R^{n-1})\bigr)\cong \mathbb Z\oplus\mathbb Z
	\]
	the index equals $(1,1)\neq 0$. One can see this by pairing the Dirac operator with the almost flat Bott classes of the two copies of $\R^{n-1}$ independently, using representatives that are constant near the tube between them.
	
	Nevertheless, by taking $R$ arbitrarily large, the distance between $\partial_-M$ and $\partial_+M$ can be made arbitrarily long while maintaining positive scalar curvature lower bound and non-zero index. In this example, all but the hypothesis 1 of Theorem \ref{thm noncpt bandwidth} and Theorem \ref{thm:bandwidthForBand} are satisfied, that is, the identity map
	\[
	\id \colon (\partial_-M,d_{g_{\partial_-M}})\longrightarrow (\partial_-M,d_M)
	\]
	is not a coarse equivalence. (Indeed, $ (\partial_-M,d_M)$ is coarsely equivalent to $\R^{n-1}$, which is not coarsely equivalent to $\R^{n-1} \# \R^{n-1}$.)

\bibliography{eta}

\begin{thebibliography}{10}

\bibitem{MR715325}
S.~Baaj and P.~Julg.
\newblock Th\'eorie bivariante de {K}asparov et op\'erateurs non born\'es dans
  les {$C\sp{\ast} $}-modules hilbertiens.
\newblock {\em C. R. Acad. Sci. Paris S\'er. I Math.}, 296(21):875--878, 1983.

\bibitem{Blackadar}
B.~Blackadar.
\newblock {\em {$K$}-theory for operator algebras}, volume~5 of {\em
  Mathematical Sciences Research Institute Publications}.
\newblock Cambridge University Press, Cambridge, second edition, 1998.

\bibitem{bunke2024coronas}
U.~Bunke and M.~Ludewig.
\newblock Coronas and {C}allias type operators in coarse geometry.
\newblock arXiv:2411.01646, 2024.

\bibitem{Cecchini20}
S.~Cecchini.
\newblock A long neck principle for {R}iemannian spin manifolds with positive
  scalar curvature.
\newblock {\em Geom. Funct. Anal.}, 30(5):1183--1223, 2020.

\bibitem{MR4937352}
J.~Choi and Y.-G. Oh.
\newblock Injectivity radius lower bound of convex sum of tame {R}iemannian
  metrics and applications to symplectic topology.
\newblock {\em Adv. Math.}, 479:Paper No. 110443, 49, 2025.

\bibitem{ConnesSkandalis}
A.~Connes and G.~Skandalis.
\newblock The longitudinal index theorem for foliations.
\newblock {\em Publ. Res. Inst. Math. Sci.}, 20(6):1139--1183, 1984.

\bibitem{Normallyhyperbolicinvariantmanifolds}
J.~Eldering.
\newblock {\em Normally hyperbolic invariant manifolds}, volume~2 of {\em
  Atlantis Studies in Dynamical Systems}.
\newblock Atlantis Press, Paris, 2013.
\newblock The noncompact case.

\bibitem{Engel2025relative}
A.~Engel and C.~Wulff.
\newblock The relative index in coarse index theory and submanifold
  obstructions to uniform positive scalar curvature.
\newblock arXiv:2506.14301, 2025.

\bibitem{Gromov18}
M.~Gromov.
\newblock Metric inequalities with scalar curvature.
\newblock {\em Geom. Funct. Anal.}, 28(3):645--726, 2018.

\bibitem{GromovLawson80}
M.~Gromov and H.~B. Lawson, Jr.
\newblock Spin and scalar curvature in the presence of a fundamental group.
  {I}.
\newblock {\em Ann. of Math. (2)}, 111(2):209--230, 1980.

\bibitem{Higson91}
N.~Higson.
\newblock A note on the cobordism invariance of the index.
\newblock {\em Topology}, 30(3):439--443, 1991.

\bibitem{Higson00}
N.~Higson and J.~Roe.
\newblock {\em Analytic {$K$}-homology}.
\newblock Oxford Mathematical Monographs. Oxford University Press, Oxford,
  2000.
\newblock Oxford Science Publications.

\bibitem{HRY93}
N.~Higson, J.~Roe, and G.~Yu.
\newblock A coarse {M}ayer-{V}ietoris principle.
\newblock {\em Math. Proc. Cambridge Philos. Soc.}, 114(1):85--97, 1993.

\bibitem{HochsdeKok25}
P.~Hochs and T.~de~Kok.
\newblock A partitioned manifold index theorem for noncompact hypersurfaces.
\newblock arXiv:2507.16591, 2025.

\bibitem{Karami2019relative-partitioned}
M.~Karami, M.~Zadeh, and A.~Sadegh.
\newblock A coarse relative-partitioned index theorem.
\newblock {\em Bull. Sci. Math.}, 153:57--71, 2019.

\bibitem{Ludewig2025large}
M.~Ludewig and G.~Thiang.
\newblock Large-scale quantization of trace {I}: Finite propagation operators.
\newblock {\em arXiv preprint arXiv:2506.10957}, 2025.

\bibitem{Roe1988dualtoeplitz}
J.~Roe.
\newblock Partitioning noncompact manifolds and the dual {T}oeplitz problem.
\newblock In {\em Operator algebras and applications, {V}ol.\ 1}, volume 135 of
  {\em London Math. Soc. Lecture Note Ser.}, pages 187--228. Cambridge Univ.
  Press, Cambridge, 1988.

\bibitem{Roe1996indextheory}
J.~Roe.
\newblock {\em Index theory, coarse geometry, and topology of manifolds},
  volume~90 of {\em CBMS Regional Conference Series in Mathematics}.
\newblock Conference Board of the Mathematical Sciences, Washington, DC; by the
  American Mathematical Society, Providence, RI, 1996.

\bibitem{Rosenberg83}
J.~Rosenberg.
\newblock {$C^{\ast} $}-algebras, positive scalar curvature, and the {N}ovikov
  conjecture.
\newblock {\em Inst. Hautes \'{E}tudes Sci. Publ. Math.}, (58):197--212 (1984),
  1983.

\bibitem{Rosenberg86a}
J.~Rosenberg.
\newblock {$C^\ast$}-algebras, positive scalar curvature and the {N}ovikov
  conjecture. {II}.
\newblock In {\em Geometric methods in operator algebras ({K}yoto, 1983)},
  volume 123 of {\em Pitman Res. Notes Math. Ser.}, pages 341--374. Longman
  Sci. Tech., Harlow, 1986.

\bibitem{Rosenberg86b}
J.~Rosenberg.
\newblock {$C^\ast$}-algebras, positive scalar curvature, and the {N}ovikov
  conjecture. {III}.
\newblock {\em Topology}, 25(3):319--336, 1986.

\bibitem{Schick2018largescale}
T.~Schick and M.~Zadeh.
\newblock Large scale index of multi-partitioned manifolds.
\newblock {\em J. Noncommut. Geom.}, 12(2):439--456, 2018.

\bibitem{SchoenYau79}
R.~Schoen and S.~T. Yau.
\newblock On the structure of manifolds with positive scalar curvature.
\newblock {\em Manuscripta Math.}, 28(1-3):159--183, 1979.

\bibitem{Seto2018toeplitz}
T.~Seto.
\newblock Toeplitz operators and the {R}oe-{H}igson type index theorem.
\newblock {\em J. Noncommut. Geom.}, 12(2):637--670, 2018.

\bibitem{Sweeney}
P.~Sweeney, Jr.
\newblock Examples for scalar sphere stability.
\newblock {\em Calc. Var. Partial Differential Equations}, 64(6):Paper No. 188,
  33, 2025.

\bibitem{WangJinmin25}
J.~Wang.
\newblock Non-negative scalar curvature on spin surgeries and {N}ovikov
  conjecture.
\newblock arXiv:2512.16535, 2025.

\bibitem{WXYZ24}
J.~Wang, Z.~Xie, G.~Yu, and B.~Zhu.
\newblock $l^p$-coarse {B}aum--{C}onnes conjecture for $l^q$-coarse embeddable
  spaces.
\newblock arXiv:2411.15070, 2024.

\bibitem{higherindex}
R.~Willett and G.~Yu.
\newblock {\em Higher index theory}, volume 189 of {\em Cambridge Studies in
  Advanced Mathematics}.
\newblock Cambridge University Press, Cambridge, 2020.

\bibitem{Wulff12}
C.~Wulff.
\newblock Bordism invariance of the coarse index.
\newblock {\em Proc. Amer. Math. Soc.}, 140(8):2693--2697, 2012.

\bibitem{Wulff19}
C.~Wulff.
\newblock Coarse indices of twisted operators.
\newblock {\em J. Topol. Anal.}, 11(4):823--873, 2019.

\bibitem{XieYu14}
Z.~Xie and G.~Yu.
\newblock Positive scalar curvature, higher rho invariants and localization
  algebras.
\newblock {\em Adv. Math.}, 262:823--866, 2014.

\bibitem{Zadeh2010indextheory}
M.~Zadeh.
\newblock Index theory and partitioning by enlargeable hypersurfaces.
\newblock {\em J. Noncommut. Geom.}, 4(3):459--473, 2010.

\bibitem{Zeidler20}
R.~Zeidler.
\newblock Width, largeness and index theory.
\newblock {\em SIGMA Symmetry Integrability Geom. Methods Appl.}, 16:Paper No.
  127, 15, 2020.

\bibitem{Zeidler22}
R.~Zeidler.
\newblock Band width estimates via the {D}irac operator.
\newblock {\em J. Differential Geom.}, 122(1):155--183, 2022.

\end{thebibliography}

\end{document}